\documentclass[11pt]{amsart}
\usepackage{color,amssymb}
\usepackage[normalem]{ulem}
\usepackage{bbm}

\usepackage[all]{xy}

\newtheorem{theorem}{Theorem}[section]
\newtheorem{claim}[theorem]{Claim}

\newtheorem{proposition}[theorem]{Proposition}
\newtheorem{corollary}[theorem]{Corollary}
\newtheorem{lemma}[theorem]{Lemma}
\newtheorem{problem}[theorem]{Problem}

\theoremstyle{definition}
\newtheorem{remark}[theorem]{Remark}
\newtheorem{definition}[theorem]{Definition}
\newtheorem*{definition*}{Definition}

\newcommand{\IN}{\mathbb N}
\newcommand{\IR}{\mathbb R}

\newcommand{\IZ}{\mathbb Z}
\newcommand{\F}{\mathcal F}

\newcommand{\C}{\mathcal C}

\newcommand{\w}{\omega}
\newcommand{\pr}{\mathrm{pr}}

\newcommand{\Ra}{\Rightarrow}
\newcommand{\Haus}{\mathsf{T_{\!2}S}}

\newcommand{\Zero}{\mathsf{T_{\!z}S}}

\newcommand{\korin}[2]{\!\sqrt[#1]{#2}}

\usepackage{relsize}
\newcommand*{\defeq}{\stackrel{\mathsmaller{\mathsf{def}}}{=}}

\newcommand{\two}{\mathbbm 2}

\newcommand{\cov}{\mathrm{cov}}
\newcommand{\M}{\mathcal M}
\newcommand{\lcm}{\mathsf{lcm}}

\title{Absolutely closed semigroups}
\author{Taras Banakh and Serhii Bardyla}
\address{T.Banakh: Ivan Franko National University of Lviv (Ukraine) and Jan Kochanowski University in Kielce (Poland)}
\email{t.o.banakh@gmail.com}
\address{S.~Bardyla: University of Vienna, Institute of Mathematics, Kurt G\"{o}del Research Center (Austria)}
\thanks{The second author was supported by the Austrian Science Fund FWF (Grant  M 2967).}
\email{sbardyla@yahoo.com}
\makeatletter
\@namedef{subjclassname@2020}{%
  \textup{2020} Mathematics Subject Classification}
\makeatother


\begin{document}
\begin{abstract}Let $\C$ be a class of topological semigroups. A semigroup $X$ is called {\em absolutely $\C$-closed} if for any homomorphism $h:X\to Y$ to a topological semigroup $Y\in\C$, the image $h[X]$ is closed in $Y$. Let $\mathsf{T_{\!1}S}$, $\mathsf{T_{\!2}S}$, and $\mathsf{T_{\!z}S}$ be  the classes of $T_1$,  Hausdorff, and Tychonoff zero-dimensional topological semigroups, respectively. We prove that a commutative semigroup $X$ is absolutely $\mathsf{T_{\!z}S}$-closed if and only if $X$ is absolutely $\mathsf{T_{\!2}S}$-closed if and only if $X$ is chain-finite, bounded, group-finite and Clifford+finite. On the other hand, a commutative semigroup $X$ is absolutely $\mathsf{T_{\!1}S}$-closed if and only if $X$ is finite. Also, for a given absolutely $\C$-closed semigroup $X$ we detect absolutely $\C$-closed subsemigroups in the center of $X$.
\end{abstract}
\maketitle

\section{Introduction and Main Results}

In many cases,  completeness properties of various objects of General Topology or  Topological Algebra can be characterized externally as closedness in ambient objects. For example, a metric space $X$ is complete if and only if $X$ is closed in any metric space containing $X$ as a subspace. A uniform space $X$ is complete if and only if $X$ is closed in any uniform space containing $X$ as a uniform subspace. A topological group $G$ is Ra\u\i kov complete  if and only if it is closed in any topological group containing $G$ as a subgroup.

On the other hand, for topological semigroups there are no reasonable notions of (inner) completeness. Nonetheless we can define many completeness properties of semigroups via their closedness in ambient topological semigroups.

A {\em topological semigroup} is a topological space $X$ endowed
with a continuous associative binary operation $X\times X\to
X$, $(x,y)\mapsto xy$.

\begin{definition*} Let $\C$ be a class of topological semigroups.
A topological
semigroup $X$ is called
\begin{itemize}
\item {\em $\C$-closed} if for any isomorphic topological
embedding $h:X\to Y$ to a topological semigroup $Y\in\C$
the image $h[X]$ is closed in $Y$;
\item {\em injectively $\C$-closed} if for any injective continuous homomorphism $h:X\to Y$ to a topological semigroup $Y\in\C$ the image $h[X]$ is closed in $Y$;
\item {\em absolutely $\C$-closed} if for any continuous homomorphism $h:X\to Y$ to a topological semigroup $Y\in\C$ the image $h[X]$ is closed in $Y$.
\end{itemize}
\end{definition*}

For any topological semigroup we have the implications:
$$\mbox{absolutely $\C$-closed $\Ra$ injectively $\C$-closed $\Ra$ $\C$-closed}.$$

\begin{definition*} A semigroup $X$ is defined to be ({\em injectively, absolutely}) {\em $\C$-closed\/} if $X$ endowed with the discrete topology has the corresponding closedness property.
\end{definition*}

We will be interested in the (absolute, injective) $\C$-closedness for the classes:
\begin{itemize}
\item $\mathsf{T_{\!1}S}$ of topological semigroups satisfying the separation axiom $T_1$;
\item $\Haus$ of Hausdorff topological semigroups;
\item $\Zero$ of Tychonoff zero-dimensional topological
semigroups.
\end{itemize}
A topological space satisfies the separation axiom $T_1$ if all its finite subsets are closed.
A topological space is {\em zero-dimensional} if it has a base of
the topology consisting of {\em clopen} (=~closed-and-open) sets. It is well-known (and easy to see) that every zero-dimensional $T_1$ topological space is Tychonoff.

Since $\mathsf{T_{\!z}S}\subseteq\mathsf{T_{\!2}S}\subseteq\mathsf{T_{\!1}S}$, for every semigroup we have the implications:
$$
\xymatrix{
\mbox{absolutely $\mathsf{T_{\!1}S}$-closed}\ar@{=>}[r]\ar@{=>}[d]&\mbox{absolutely $\mathsf{T_{\!2}S}$-closed}\ar@{=>}[r]\ar@{=>}[d]&\mbox{absolutely $\mathsf{T_{\!z}S}$-closed}\ar@{=>}[d]\\
\mbox{injectively $\mathsf{T_{\!1}S}$-closed}\ar@{=>}[r]\ar@{=>}[d]&\mbox{injectively $\mathsf{T_{\!2}S}$-closed}\ar@{=>}[r]\ar@{=>}[d]&\mbox{injectively $\mathsf{T_{\!z}S}$-closed}\ar@{=>}[d]\\
\mbox{$\mathsf{T_{\!1}S}$-closed}\ar@{=>}[r]&\mbox{$\mathsf{T_{\!2}S}$-closed}\ar@{=>}[r]&\mbox{$\mathsf{T_{\!z}S}$-closed.}
}
$$

$\C$-Closed topological groups for various classes $\C$ were investigated by many authors~\cite{AC1,BL,Ban,BGR,DU,G,JS, L}. In particular, the closedness of commutative topological groups in the class of Hausdorff topological semigroups was investigated in~\cite{Z1,Z2}; $\mathcal{C}$-closed topological semilattices were investigated in~\cite{BBm, BBc, GutikPagonRepovs2010, GutikRepovs2008, Stepp75}. Some notions of completeness in Category Theory were investigated in~\cite{CDT,Cbook1,Er,FG,G1,GH,LW}. In particular, closure operators in different categories were studied in~\cite{BGH, Cbook, CG1, CG2, DT, Dbook, GS, T, Za}. 
This paper is a continuation of the papers  \cite{BB}, \cite{BB2}, \cite{GCCS},  \cite{BV} providing inner characterizations of various closedness properties of (discrete topological) semigroups. In order to formulate such inner characterizations, let us recall some properties of semigroups.

A semigroup $X$ is called
\begin{itemize}
\item {\em chain-finite} if any infinite set $I\subseteq X$ contains elements $x,y\in I$ such that $xy\notin\{x,y\}$;
\item {\em singular} if there exists an infinite set $A\subseteq X$ such that $AA$ is a singleton;
\item {\em periodic} if for every $x\in X$ there exists $n\in\IN$ such that $x^n$ is an idempotent;
\item {\em bounded} if there exists $n\in\IN$ such that for every $x\in X$ the $n$-th power $x^n$ is an idempotent;
\item {\em group-finite} if every subgroup of $X$ is finite;
\item {\em group-bounded} if every subgroup of $X$ is bounded;
\item {\em group-commutative} if every subgroup of $X$ is commutative.
\end{itemize}

The following theorem (proved in \cite{BB}) characterizes $\C$-closed commutative semigroups.

\begin{theorem}[Banakh--Bardyla]\label{t:C-closed} Let $\C$ be a class of topological semigroups such that $\mathsf{T_{\!z}S}\subseteq\C\subseteq \mathsf{T_{\!1}S}$. A commutative semigroup $X$ is $\C$-closed if and only if $X$ is chain-finite, nonsingular,  periodic, and group-bounded.
\end{theorem}

A subset $I$ of a semigroup $X$ is called an {\em ideal} in $X$ if $IX\cup XI\subseteq  I$. Every ideal $I\subseteq X$ determines the congruence $(I\times I)\cup \{(x,y)\in X\times X:x=y\}$ on $X$. The quotient semigroup of $X$ by this congruence is denoted by $X/I$ and called the {\em quotient semigroup} of $X$ by the ideal $I$. If $I=\emptyset$, then the quotient semigroup $X/\emptyset$ can be identified with the semigroup $X$.

Theorem~\ref{t:C-closed} implies that each subsemigroup of a $\C$-closed commutative semigroup is $\C$-closed. On the other hand, quotient semigroups of $\C$-closed commutative semigroups are not necessarily $\C$-closed, see Example 1.8 in \cite{BB}. This motivates the following notions.

\begin{definition*}A semigroup $X$ is called
\begin{itemize}
\item {\em projectively $\C$-closed} if for any congruence $\approx$ on $X$ the quotient semigroup $X/_{\approx}$ is $\C$-closed;
\item {\em ideally $\C$-closed} if for any ideal $I\subseteq X$ the quotient semigroup $X/I$ is $\C$-closed.
\end{itemize}
\end{definition*}

It is easy to see that for every semigroup the following implications hold:
$$\mbox{absolutely $\C$-closed $\Ra$ projectively $\C$-closed $\Ra$ ideally $\C$-closed $\Ra$ $\C$-closed.}$$
Observe that a semigroup $X$ is absolutely $\C$-closed if and only if for any congruence $\approx$ on $X$ the semigroup $X/_\approx$ is injectively $\C$-closed.

For a semigroup $X$, let $$E(X)\defeq\{x\in X:xx=x\}$$ be the set of idempotents of $X$.

For an idempotent $e$ of a semigroup $X$, let $H_e$ be the maximal subgroup of $X$ that contains $e$. The union $H(X)=\bigcup_{e\in E(X)}H_e$ of all subgroups of $X$ is called the {\em Clifford part} of $S$.
A semigroup $X$ is called
\begin{itemize}
\item {\em Clifford}  if $X=H(X)$;
\item {\em Clifford+finite} if $X\setminus H(X)$ is finite.
\end{itemize}

Ideally and projectively $\C$-closed commutative semigroups were characterized in  \cite{BB} as follows.

\begin{theorem}[Banakh--Bardyla]\label{t:mainP} Let $\C$ be a class of topological semigroups such that $\mathsf{T_{\!z}S}\subseteq\C\subseteq \mathsf{T_{\!1}S}$. For a commutative semigroup $X$ the following conditions are equivalent:
\begin{enumerate}
\item $X$ is projectively $\C$-closed;
\item $X$ is ideally $\C$-closed;
\item the semigroup $X$ is chain-finite, group-bounded and Clifford+finite.
\end{enumerate}
\end{theorem}

In \cite{BB2} it is shown that the injective (and absolute) $\mathsf{T_{\!1}S}$-closedness is tightly related to the (projective) $\mathsf{T_{\!1}S}$-discreteness.

\begin{definition*}
Let $\C$ be a class of topological semigroups.
A semigroup $X$ is called
\begin{itemize}
\item {\em $\C$-discrete} (or else {\em $\C$-nontopologizable}) if for any injective homomorphism $h:X\to Y$ to a topological semigroup $Y\in\C$ the image $h[X]$ is a discrete subspace of $Y$;
\item {\em $\C$-topologizable} if $X$ is not $\C$-discrete;
\item {\em projectively $\C$-discrete} if for every homomorphism $h:X\to Y$ to a topological semigroup $Y\in\C$ the image $h[X]$ is a discrete subspace of $Y$.
\end{itemize}
\end{definition*}

The study of topologizable and nontopologizable semigroups is a classical topic in Topological Algebra that traces its history back to Markov's problem \cite{Markov} of topologizability of infinite groups, which was resolved in \cite{Shelah}, \cite{Hesse} and \cite{Ol} by constructing examples of nontopologizable infinite groups. For some other results on topologizability of semigroups, see \cite{BanS,BM,DS1,DS2,DT0,vD,DI,GLS,KOO,Kotov,Taimanov}.

In Propositions 3.2 and 3.3 of \cite{BB2} the following two characterizations are proved.

\begin{theorem}[Banakh--Bardyla]\label{t:discrete} A semigroup $X$ is
\begin{enumerate}
\item injectively $\mathsf{T_{\!1}S}$-closed if and only if $X$ is $\mathsf{T_{\!1}S}$-closed and $\mathsf{T_{\!1}S}$-discrete;
\item absolutely $\mathsf{T_{\!1}S}$-closed if and only if $X$ is projectively $\mathsf{T_{\!1}S}$-closed and projectively $\mathsf{T_{\!1}S}$-discrete.
\end{enumerate}
\end{theorem}

The following two theorems characterizing absolutely $\C$-closed commutative semigroups are the main results of this paper. In contrast to  Theorems~\ref{t:C-closed} and~\ref{t:mainP}, the  characterizations of absolutely $\C$-closed semigroups essentially depend on the class $\C$, where we distinguish two cases: $\C=\mathsf{T_{\!1}S}$ and $\mathsf{T_{\!z}S}\subseteq \C\subseteq \mathsf{T_{\!2}S}$.

\begin{theorem}\label{t:main1} For a commutative semigroup $X$ the following conditions are equivalent:
\begin{enumerate}
\item $X$ is absolutely $\mathsf{T_{\!1}S}$-closed;
\item $X$ is projectively $\mathsf{T_{\!1}S}$-closed and projectively $\mathsf{T_{\!1}S}$-discrete;
\item $X$ is projectively $\mathsf{T_{\!z}S}$-closed and projectively $\mathsf{T_{\!z}S}$-discrete;
\item $X$ is finite.
\end{enumerate}
\end{theorem}

\begin{theorem}\label{t:main2} Let $\C$ be a class of topological semigroups such that $\mathsf{T_{\!z}S}\subseteq\C\subseteq\mathsf{T_{\!2}S}$. For a commutative semigroup $X$ the following conditions are equivalent:
\begin{enumerate}
\item $X$ is absolutely $\C$-closed;
\item $X$ is ideally $\C$-closed, injectively $\C$-closed and bounded;
\item $X$ is ideally $\C$-closed, group-finite and bounded;
\item $X$ is chain-finite, bounded, group-finite and Clifford+finite.
\end{enumerate}
\end{theorem}

Theorems~\ref{t:main1} and \ref{t:main2} imply that the absolute $\C$-closedness of commutative semigroups is inherited by subsemigroups:

\begin{corollary}\label{c:her} Let $\C$ be a class of topological semigroups such that either $\C=\mathsf{T_{\!1}S}$ or $\mathsf{T_{\!z}S}\subseteq\C\subseteq\mathsf{T_{\!2}S}$. Every subsemigroup of an absolutely $\C$-closed commutative semigroup is absolutely $\C$-closed.
\end{corollary}

\begin{remark} Corollary~\ref{c:her}  does not generalize to noncommutative groups: by Theorem 1.17 in \cite{BB2},  every countable bounded group $G$ without elements of order 2 is a subgroup of an absolutely $\mathsf{T_{\!1}S}$-closed countable simple bounded group $X$. If the group $G$ has infinite center, then $G$ is not injectively $\mathsf{T_{\!z}S}$-closed by Theorem~\ref{t:Z}(2) below. On the other hand, $G$ is a subgroup of the absolutely $\mathsf{T_{\!1}S}$-closed group $X$. This example also shows that the equivalences $(1)\Leftrightarrow(4)$ in Theorems~\ref{t:main1} and \ref{t:main2} do not hold for non-commutative groups.
\end{remark}

For a semigroup $X$ let
$$Z(X)\defeq\{z\in X:\forall x\in X\;\;(zx=xz)\}$$be the {\em center} of $X$,
 and
$$I\!Z(X)\defeq\{z\in Z(X):zX\subseteq Z(X)\}$$ be the {\em ideal center} of $X$. Every commutative semigroup $X$ has $I\!Z(X)=Z(X)=X$.

The following theorem proved in \cite[\S5]{BB} and \cite{GCCS} describes some properties of the center of a semigroup possessing various closedness properties.

\begin{theorem}[Banakh-Bardyla]\label{t:Z} Let $X$ be a semigroup.
\begin{enumerate}
\item If $X$ is $\mathsf{T_{\!z}S}$-closed, then the center $Z(X)$ is chain-finite, periodic and nonsingular.
\item If $X$ is $\mathsf{T_{\!z}S}$-discrete or injectively $\mathsf{T_{\!z}S}$-closed, then $Z(X)$ is group-finite.
\item If $X$ is ideally $\mathsf{T_{\!z}S}$-closed, then $Z(X)$ is group-bounded.
\end{enumerate}
\end{theorem}

In~\cite{GCCS} it was proved that the (ideal) $\C$-closedness is inherited by the ideal center:

\begin{theorem}[Banakh-Bardyla]\label{t:ideal-center} Let $\C$ be a class of topological semigroups such that $\mathsf{T_{\!z}S}\subseteq\C\subseteq\mathsf{T_{\!1}S}$.
For any (ideally) $\C$-closed semigroup $X$, its ideal center $I\!Z(X)$ is (ideally) $\C$-closed.
\end{theorem}

Theorem~\ref{t:ideal-center} suggests the following problem.

\begin{problem}\label{prob} Let $\C$ be a class of topological semigroups. Is the (ideal) center of any absolutely $\C$-closed semigroup $X$ absolutely $\C$-closed?
\end{problem}

The ``ideal'' version of Problem~\ref{prob} has an affirmative answer.

\begin{theorem}\label{t:ac=>ideal} Let $\C$ be a class of topological semigroups such that either $\C=\mathsf{T_{\!1}S}$ or $\mathsf{T_{\!z}S}\subseteq\C\subseteq\mathsf{T_{\!2}S}$. Every absolutely $\C$-closed semigroup $X$ has absolutely $\C$-closed ideal center $I\!Z(X)$.
\end{theorem}

The ``non-ideal'' version of Problem~\ref{prob} has an affirmative answer for  group-commutative $Z$-viable semigroups.

Following Putcha and Weissglass \cite{PW} we call semigroup $X$ {\em viable} if for any $x,y\in X$ with  $\{xy,yx\}\subseteq E(X)$ we have $xy=yx$. This notion can be localized using the notion of a viable idempotent.

An idempotent $e$ in a semigroup $X$ is defined to be {\em viable} if the set
$$\tfrac{H_e}e\defeq\{x\in X:xe=ex\in H_e\}$$
is a {\em coideal} in $X$ in the sense that $X\setminus\frac{H_e}e$ is an ideal in $X$. By $V\!E(X)$ we denote the set of viable idempotents of a semigroup $X$. 

By Theorem~3.2 of \cite{BanE}, a semigroup $X$ is viable if and only if each idempotent of $X$ is viable if and only if for every $x,y\in X$ with $xy=e\in E(X)$ we have $xe=ex$ and $ye=ey$. This characterization implies that every semigroup $X$ with $E(X)\subseteq Z(X)$ is viable. In particular, every commutative semigroup is viable.

For ideally (absolutely) $\mathsf{T_{\!z}S}$-closed semigroups we have the following description of the structure of maximal subgroups of viable idempotents, see \cite[Theorem 1.7]{GCCS}.

\begin{theorem}[Banakh-Bardyla]\label{t:G} Let $e$ be a viable idempotent of a semigroup $X$ and $H_e$ be the maximal subgroup of $e$ in $X$.
\begin{enumerate}
\item If $X$ is ideally $\mathsf{T_{\!z}S}$-closed, then the group $Z(H_e)$ is bounded.
\item If $X$ is absolutely $\mathsf{T_{\!z}S}$-closed, then the group $Z(H_e)$ is finite.
\end{enumerate}
\end{theorem}

A semigroup $X$ is called {\em $Z$-viable} if $Z(X)\cap E(X)\subseteq V\!E(X)$, i.e., if each central idempotent of $X$ is viable. It is clear that each viable semigroup is $Z$-viable. On the other hand, there exist semigroups which are not $Z$-viable, see Remark~\ref{rem26}.

For a subset $A$ of semigroup $X$ let
$$\korin{\IN}{\!A}\defeq\bigcup_{n\in\IN}\korin{n}{\!A}\quad\mbox{where}\quad\korin{n}{\!A}\defeq\{x\in X:x^n\in A\}.$$
A subset $B$ of a semigroup $X$ is called {\em bounded} if $B\subseteq\korin{n}{\!E(X)}$ for some $n\in\IN$.
In fact, the ``ideal'' part of Theorem~\ref{t:ideal-center} was derived in \cite{GCCS} from the following theorem, which will be essentially used also in this paper:

\begin{theorem}[Banakh--Bardyla]\label{t:ideal} If a semigroup $X$ is ideally $\mathsf{T_{\!z}S}$-closed, then the set\\ $Z(X)\cap\korin{\IN}{V\!E(X)}\setminus H(X)$ is finite.
\end{theorem}

The following theorem gives a partial answer to Problem~\ref{prob} for the class $\C=\mathsf{T_{\!1}S}$.

\begin{theorem} If a semigroup $X$ is absolutely $\mathsf{T_{\!1}S}$-closed (and $Z$-viable), then the set $Z(X)\cap\korin{\IN}{V\!E(X)}$ is finite (and the semigroup $Z(X)$  is absolutely $\mathsf{T_{\!1}S}$-closed).
\end{theorem}

For classes $\C$ with $\mathsf{T_{\!z}S}\subseteq\C\subseteq\mathsf{T_{\!2}S}$ a partial answer to Problem~\ref{prob} looks as follows.

\begin{theorem}\label{t:subsem} Let $X$ be an absolutely $\mathsf{T_{\!z}S}$-closed semigroup and $A\subseteq  V\!E(X)$. Assume that for any infinite countable subset $B\subseteq A$ and the subsemigroup $C\defeq\bigcap_{e\in B}\frac{H_e}e$ of $X$, one of the following conditions is satisfied:
\begin{enumerate}
\item for every $e\in B$ the subsemigroup $Ce$ of $H_e$ is commutative;
\item $C$ is countable;
\item $|C|\le\cov(\M)$, and for every $e\in A$ the subsemigroup $Ce$ of $H_e$ is countable.
\item $|C|\le\mathfrak c$ and for every $e\in A$ the subsemigroup $Ce$ of $H_e$ is bounded.
\end{enumerate}
Then the set $Z(X)\cap\korin{\IN}{\!A}$ is bounded, and every subsemigroup of $S\subseteq Z(X)\cap\korin{\IN}{\!A}$ of $X$ is  absolutely $\mathsf{T_{\!2}S}$-closed.
\end{theorem}
The cardinal $\cov(\M)$ appearing in Theorem~\ref{t:subsem}(3) is defined as the smallest cardinality of a cover of the real line by nowhere dense subsets. The Baire Theorem implies that $\w_1\le\cov(\M)\le \mathfrak c$. It is well-known that $\cov(\M)=\mathfrak c$ under Martin's Axiom. By \cite[7.13]{Blass}, the equality $\cov(\M)=\mathfrak c$ is equivalent to Martin's Axiom for countable posets.

By Theorem~\ref{t:Z}(1), the center $Z(X)$ of any $\mathsf{T_{\!z}S}$-closed semigroup is chain-finite. In fact, this is an order property of the poset $E(X)$ endowed with the natural partial order $\le$ defined by $x\le y$ iff $xy=yx=x$. In turns out that stronger closedness properties (like the ideal or projective $\C$-closedness) impose stronger restrictions on the partial order of the set $E(X)$ and also on the partial order of the semilattice reflection $X/_{\Updownarrow}$ of $X$.

A congruence $\approx$ on a semigroup $X$ is called a {\em semilattice congruence} if the quotient semigroup $X/_\approx$ is a {\em semilattice}, i.e., a commutative semigroup of idempotents. The intersection ${\Updownarrow}$ of all semilattice congruences on a semigroup $X$ is called the {\em smallest semilattice congruence} on $X$ and the quotient semigroup $X/_{\Updownarrow}$ is called the {\em semilattice reflection} of $X$. The smallest semilattice congruence was studied in the monographs \cite{BCP,Mitro2003}, surveys \cite{Mitro2004,MS} and papers \cite{Petrich63,Petrich64,Putcha73,PW,Tam56,Tamura82,TS66,Ban}.

A partially ordered set $(P,\le)$ is called
\begin{itemize}
\item {\em chain-finite} if each infinite subset $I\subseteq P$ contains two elements $x,y\in I$ such that $x\not\le y$ and $y\not
\le x$;
\item {\em well-founded} if each nonempty set $A\subseteq P$ contains an element $a\in A$ such that $\{x\in A:x\le a\}=\{a\}$.
\end{itemize}
It is easy to see that for every chain-finite semigroup $X$ the poset $E(X)$ is chain-finite. The converse holds if $E(X)$ is a commutative subsemigroup of $X$.

\begin{theorem}\label{t:order} Let $X$ be a semigroup.
\begin{enumerate}
\item If $X$ is ideally $\mathsf{T_{\!z}S}$-closed, then the posets   $X/_{\Updownarrow}$ and $V\!E(x)$ are well-founded.
\item If $X$ is projectively $\mathsf{T_{\!z}S}$-closed, then  $X/_{\Updownarrow}$ and $V\!E(x)$ are chain-finite.
\item If $X$ is projectively $\mathsf{T_{\!z}S}$-closed and projectively $\mathsf{T_{\!z}S}$-discrete, then  $X/_{\Updownarrow}$ and $V\!E(x)$ are finite;
\item If $X$ is absolutely $\mathsf{T_{\!1}S}$-closed, then  $X/_{\Updownarrow}$ and $V\!E(x)$ are finite.
\end{enumerate}
\end{theorem}

Theorem~\ref{t:order} will be proved in Section~\ref{s:order}. In Section~\ref{s:T1} we prove a general version of Theorem~\ref{t:main1} and in Section~\ref{s:if} we prove Lemma~\ref{l:chain-closed2} giving a sufficient condition of the absolute $\mathsf{T_{\!2}S}$-closedness. In Section~\ref{s:bounded} we introduce the notion of an $A$-centrobounded semigroup and use this notion for characterizing bounded set of form $Z(X)\cap\korin{\IN}{\!A}$ in absolutely $\mathsf{T_{\!z}S}$-closed semigroups. In Section~\ref{s:Acb} we prove Theorem~\ref{t:=>Acb} giving some sufficient conditions of the $A$-centroboundedness and implying Corollary~\ref{c:subsem}, which is a more general version of Theorem~\ref{t:subsem}. In Sections~\ref{s:ideal} and \ref{s:final} we prove Theorems~\ref{t:ac=>ideal} and \ref{t:main2}, respectively.

\section{Preliminaries}

In this section we collect some auxiliary results and notions that will be used in the remaining part of the paper.

We denote by $\w$ the set of finite ordinals and by $\IN\defeq \w\setminus\{0\}$ the set of positive integers. Each ordinal $n\in\w$ is identified with the set $\{k:k<n\}$ of smaller ordinals.

\subsection{Partially ordered sets} A {\em poset} is a set endowed with a partial order $\le$. For an element $p$ of a poset $P$, let
$${\downarrow}p\defeq\{x\in P:x\le p\}\quad\mbox{and}\quad{\uparrow}p\defeq\{x\in P:p\le x\}$$be the {\em lower} and {\em upper sets} of $p$ in $P$, respectively.

\subsection{Cardinal characteristics of the continuum}

Let
\begin{itemize}
\item $\cov(\M)$ be the smallest cardinality of a cover of the real line by nowhere dense sets,
\item $\cov(\mathcal N)$ be the smallest cardinality of a cover of the real line by sets of Lebesgue measure zero,
\item  $\cov(\overline{\mathcal N})$ be the smallest cardinality of a cover of the real line by closed subsets of Lebesgue measure zero, and
\item $\mathfrak d$ be the smallest cardinality of a cover of the Baire space $\w^\w$ by compact sets.
\end{itemize}

By \cite[4.1]{BS}, $$\max\{\cov(\M),\cov(\mathcal N)\}\le\cov(\overline{\mathcal N})\le\max\{\cov(\mathcal N),\mathfrak d\}\le\mathfrak c.$$ Martin's Axiom implies that $\mathfrak d=\cov(\M)=\cov(\mathcal N)=\cov(\overline{\mathcal N})=\mathfrak c$, see \cite[\S7]{Blass}. By Theorem 7.13 in \cite{Blass}, the equality $\cov(\M)=\mathfrak c$ is equivalent to the Martin's Axiom for countable posets. By \cite[5.6]{BS} and \cite[11.5]{Blass}, the strict inequalities $\max\{\cov(\M),\cov(\mathcal N)\}<\cov(\overline{\mathcal N})$ and $\max\{\cov(\mathcal N),\mathfrak d\}<\mathfrak c$ are consistent.

\subsection{Semigroup topologies and subinvariant metrics on semigroups}

A topology $\tau$ on a semigroup $X$ is called a {\em semigroup topology} if $(X,\tau)$ is a topological semigroup.

A metric $d$ on a semigroup $X$ is {\em subinvariant} if for every $x,y,a\in X$ and  we have $$d(ax,ay)\le d(x,y)\quad\mbox{and}\quad d(xa,ya)\le d(x,y).$$

It is easy to see that every subinvariant metric on a semigroup generates a semigroup topology.

\subsection{Zero-closed semigroups}

For a semigroup $X$ its
\begin{itemize}
\item {\em $0$-extension} is the semigroup $X^0=X\cup\{0\}$ where $0\notin X$ is any element such that $0x=0=x0$ for every $x\in X^0$;
\item {\em $1$-extension} is the semigroup $X^1=X\cup\{1\}$ where $1\notin X$ is any element such that $1x=x=x1$ for every  $x\in X^1$.
\end{itemize}

Following \cite{BB2}, we call a semigroup $X$ {\em zero-closed} if $X$ is closed in its $0$-extension $X^0=\{0\}\cup X$ endowed with any Hausdorff semigroup topology.

A topological semigroup $X$ is called {\em $0$-discrete} if $X$ contains a unique non-isolated point $0\in X$ such that $x0=0=0x$ for all $x\in X$. It is easy to see that every $0$-discrete $T_1$ topological semigroup is zero-dimensional.

\begin{lemma}\label{l:c=>zc} Let $\C$ be a class of topological semigroups containing all $0$-discrete semigroups. Every $\C$-closed semigroup $X$ is zero-closed.
\end{lemma}

\begin{proof} Assuming that $X$ is not zero-closed, we can find a Hausdorff semigroup topology $\tau$ on $X^0$ such that $X$ is not closed in the topological space $(X^0,\tau)$. Consider the topology $\tau^0$ on $X^0$, generated by the base $\big\{\{x\}:x\in X\big\}\cup\tau$, and observe that $(X^0,\tau^0)$ is a $0$-discrete Hausdorff topological semigroup containing $X$ as a non-closed subsemigroup and witnessing that $X$ is not $\C$-closed.
\end{proof}

\subsection{Polybounded and polyfinite semigroups}

A {\em semigroup polynomial} on a semigroup $X$ is a function $f:X\to X$ of the form $f(x)=a_0xa_1x\cdots xa_n$ for some $n\in\IN$  and some elements $a_0,\dots, a_n\in X^1$. The number $n$ is called the {\em degree} of the polynomial $f$ and is denoted by $\deg(f)$.

A semigroup $X$ is called {\em $\kappa$-polybounded} for a cardinal $\kappa$ if $X=\bigcup_{\alpha\in\kappa}f_\alpha^{-1}(b_\alpha)$ for some elements $b_\alpha\in X$ and semigroup polynomials $f_\alpha$ on $X$. A semigroup $X$ is {\em polybounded} if $X$ is $n$-polybounded for some $n\in\IN$.

 Polybounded semigroups were introduced in \cite{BB2}, where it was proved that  countable zero-closed semigroups are polybounded and polybounded groups are absolutely $\mathsf{T_{\!1}S}$-closed.

A semigroup $X$ is called {\em polyfinite} if there exist $d\in\IN$ and a finite set $F\subseteq X$ such that for any $x,y\in X$ there exists a semigroup polynomial $f:X\to X$ of degree $\le d$ such that $\{f(x),f(y)\}\subseteq F$.

\begin{lemma}\label{l:polyb=>polyf} Every polybounded semigroup $X$ is polyfinite.
\end{lemma}

\begin{proof} Since $X$ is polybounded, there exist elements $b_0,\dots,b_{n-1}\in X$ and semigroup polynomials $f_0,\dots,f_{n-1}$ on $X$ such that $X=\bigcap_{i\in n}f_i^{-1}(b_i)$. Let $$F=\{b_i\}_{i\in n}\cup\{f_i(b_j):i,j\in n\}\quad\mbox{and}\quad d=\max\{\deg(f_i\circ f_j):i,j\in n\}.$$ Given any elements $x,y\in X$, find $i\in n$ such that $f_i(x)=b_i$ and then find $j\in n$ such that $f_j(f_i(y))=b_j$. The semigroup polynomial $f=f_j\circ f_i:X\to X$ has degree $\le d$ and $\{f(x),f(y)\}=\{f_j(f_i(x)),f_j(f_i(y))\}=\{f_j(b_i),b_j\}\subseteq F$, witnessing that $X$ is polyfinite.
\end{proof}

  The following theorem was proved in \cite{ZeroClosed}.

\begin{theorem}\label{t:polyb} Let $X$ be a zero-closed semigroup. Then
\begin{enumerate}
\item $X$ is $\kappa$-polybounded for some $\kappa<\max\{2,|X|\}$.
\item If $\cov(\M)=\mathfrak c$ and $X$ admits a subinvariant separable complete metric, then $X$ is polybounded.
\item If $\cov(\M)=\mathfrak c$ and $X$ admits a compact Hausdorff semigroup topology, then $X$ is polybounded.
\item If $\cov(\overline{\mathcal N})=\mathfrak c$ and $X$ admits a compact Hausdorff semigroup topology, then $X$ is polyfinite.
\end{enumerate}
\end{theorem}

\subsection{Prime coideals in semigroups}

A subset $C$ of a semigroup $X$ is called a ({\em prime}) {\em coideal} if $X\setminus C$ is an ideal in $X$ (and $C$ is a subsemigroup of $X$). A subset $C\subseteq X$ is a prime coideal in $X$ if and only if its characteristic function
$$\chi_C:X\to\{0,1\},\quad\chi_C:x\mapsto\begin{cases}1&\mbox{if $x\in C$};\\
0&\mbox{if $x\in X\setminus C$};
\end{cases}
$$is a homomorphism from $X$ to the semilattice $\{0,1\}$ endowed with the operation of minimum.

\begin{lemma}\label{l:pci} If a semigroup $X$ is absolutely (resp. projectively) $\mathsf{T_{\!z}S}$-closed, then any prime coideal in $X$ is absolutely (resp. projectively) $\mathsf{T_{\!z}S}$-closed.
\end{lemma}

\begin{proof} Assume that a semigroup $X$ is absolutely (resp. projectively) $\mathsf{T_{\!z}S}$-closed and let $C$ be a prime coideal in $X$. To prove that the semigroup $C$ is absolutely (resp. projectively) $\mathsf{T_{\!z}S}$-closed, take any homomorphism $h:C\to Y$ to a topological semigroup $(Y,\tau)\in\mathsf{T_{\!z}S}$ (such that the image $h[C]$ is discrete in $Y$). Since $C$ is a prime coideal in $X$, the map
$$\bar h:X\to Y^0,\quad \bar h:x\mapsto\begin{cases}h(x)&\mbox{if $x\in C$},\\
0&\mbox{if $x\in X\setminus C$},
\end{cases}
$$is a homomorphism from $X$ to the $0$-extension $Y^0$ of the topological semigroup $Y$, endowed with the topology $\tau^0=\{U\subseteq Y^0:U\cap Y\in\tau\}$.  It follows from $(Y,\tau)\in\mathsf{T_{\!z}S}$ that $(Y^0,\tau^0)\in\mathsf{T_{\!z}S}$. By the absolute (resp. projective) $\mathsf{T_{\!z}S}$-closedness of $X$, the image $\bar h[X]$ is closed in $(Y^0,\tau^0)$ and then the set $h[C]=h[X]\cap Y$ is closed in $(Y,\tau)$, witnessing that the semigroup $C$ is absolutely (resp. projectively) $\mathsf{T_{\!z}S}$-closed.
\end{proof}

\subsection{Viable idempotents in semigroups} We recall that an idempotent $e$ of a semigroup $X$ is {\em viable} if the subsemigroup $\frac{H_e}e\defeq\{x\in X:xe=ex\in H_e\}$ is a prime coideal in $X$. By $V\!E(X)$ we denote the set of viable idempotents in $X$.

The following lemma was proved in \cite[2.5]{GCCS}.

\begin{lemma}\label{l:viable} For any semigroup $X$ we have $E(I\!Z(X))=E(Z)\cap I\!Z(X)\subseteq V\!E(X)$.
\end{lemma}

\begin{remark}\label{rem26}
The inclusion $E(Z)\cap I\!Z(X)\subseteq V\!E(X)$ in Lemma~\ref{l:viable} cannot be improved to the inclusion $E(X)\cap Z(X)\subseteq V\!E(X)$: by \cite{BG16} or \cite{CM} there exist infinite congruence-free monoids. In every  congruence-free monoid $X\ne\{1\}$ the idempotent $1$ is central but not viable.
\end{remark}

\section{Proof of Theorem~\ref{t:order}}\label{s:order}

In this section, for any semigroup $X$ we study the order properties of the posets $V\!E(X)$ and $X/_{\Updownarrow}$ and prove Theorem~\ref{t:order}.
By $\two$ we denote the two-element semilattice $\{0,1\}$ endowed with the operation of minimum.

\begin{proposition}\label{p:VE-order} Let $X$ be a semigroup and $q:X\to X/_{\Updownarrow}$ be the quotient homomorphism onto its semilattice reflection. The restriction $q{\restriction}_{V\!E(X)}$ is injective and hence is an isomorphic embedding of the poset $V\!E(X)$ into the poset $X/_{\Updownarrow}$.
\end{proposition}

\begin{proof} Given two viable idempotents $e_1,e_2\in V\!E(X)$, assume that $q(e_1)=q(e_2)$. For every $i\in\{1,2\}$, the definition of a viable idempotent ensures that the semigroup $\frac{H_{e_i}}{e_i}=\{x\in X:xe_i=e_ix\in H_{e_i}\}$ is a coideal in $X$. Then the map $h_i:X\to \two$ defined by
$$h_i(x)=\begin{cases}1&\mbox{if $x\in\frac{H_{e_i}}{e_i}$},\\
0,&\mbox{otherwise},
\end{cases}
$$is a homomorphism. The equality $q(e_1)=q(e_2)$ implies that $$h_1(e_{2})=h_1(e_{1})=1= h_2(e_{2})=h_2(e_{1}).$$ Thus, $e_1e_2=e_2e_1\in H_{e_1}\cap H_{e_2}$, which implies $e_1=e_2$, and witnesses that the restriction $q{\restriction}_{V\!E(X)}$ is injective.
\end{proof}

In the following four lemmas we prove the statements of Theorem~\ref{t:order}.

\begin{lemma} For any ideally $\mathsf{T_{\!z}S}$-closed  semigroup $X$, the posets $X/_{\Updownarrow}$ and $V\!E(X)$ are well-founded.
\end{lemma}

\begin{proof} By Proposition~\ref{p:VE-order}, the poset $V\!E(X)$ embeds into the semilattice reflection $X/_{\Updownarrow}$ of $X$, so it suffices to prove that the poset $Y\defeq X/_{\Updownarrow}$ is well-founded. Assuming that $Y$ is not well-founded, we can find a strictly decreasing sequence $(y_n)_{n\in\omega}$ in $Y$. For every $n\in\w$ consider the upper set ${\uparrow}y_n=\{y\in Y:y_n\le y\}$  and observe that ${\uparrow}y_n$ is a prime coideal in $Y$. Consequently, its preimage $P_n=q^{-1}[{\uparrow}y_n]$ is a prime coideal in $X$.

It is easy to see that $P=\bigcup_{n\in\w}P_n$ is a subsemigroup of $X$ and the complement $I\defeq X\setminus P$ is an ideal in $X$. Consider the semigroup $S=P\cup\{P\}\cup\{I\}$ endowed with the semigroup operation $*:S\times S\to S$ defined by
$$x*y=\begin{cases}xy&\mbox{if $(x,y)\in P\times P$;}\\
I&\mbox{if $(x,y)\in (S\times\{I\})\cup(\{I\}\times S)$};\\
P&\mbox{otherwise}.
\end{cases}
$$
Endow the semigroup $S$ with the topology $\tau$ generated by the base
$$\big\{\{P\}\cup (P\setminus P_n):n\in\w\big\}\cup\big\{\{x\}:x\in P\cup\{I\}\big\}$$ and observe that $(S,\tau)$ is a Hausdorff zero-dimensional topological semigroup with a unique non-isolated point $P$. Since $(S,\tau)$ contains the quotient semigroup $X/I=P\cup\{I\}$ as a discrete subsemigroup, the semigroup $X$ is not ideally $\mathsf{T_{\!z}S}$-closed, which contradicts our assumption.
\end{proof}

\begin{lemma}\label{l:pc=>cf} For any projectively $\mathsf{T_{\!z}S}$-closed  semigroup $X$, the posets $X/_{\Updownarrow}$ and $V\!E(X)$ are chain-finite.
\end{lemma}

\begin{proof} Let $q:X\to X/_{\Updownarrow}$ be the quotient homomorphism of $X$ onto its semilattice reflection. If the semigroup $X$ is projectively $\mathsf{T_{\!z}S}$-closed, then its semilattice reflection $X/_{\Updownarrow}$ is projectively $\mathsf{T_{\!z}S}$-closed and hence $\mathsf{T_{\!z}S}$-closed. By Theorem~\ref{t:C-closed}, the semilattice   $X/_{\Updownarrow}$ is chain-finite. Then $X/_{\Updownarrow}$ is also chain-finite as a poset. By Proposition~\ref{p:VE-order}, the poset $V\!E(X)$ is chain-finite, being order isomorphic to a subset of the chain-finite poset $X/_{\Updownarrow}$.
\end{proof}

\begin{lemma}\label{l:pd=>fin} If a semigroup $X$ is projectively $\mathsf{T_{\!z}S}$-closed and projectively $\mathsf{T_{\!z}S}$-discrete, then the sets  $X/_{\Updownarrow}$ and $V\!E(X)$ are finite.
\end{lemma}

\begin{proof} Let $q:X\to X/_{\Updownarrow}$ be the quotient homomorphism of $X$ onto its semilattice  reflection.
Consider the set $H$ of all homomorphisms from $X/_{\Updownarrow}$ to the two-element semilattice  $\two$. Since homomorphisms to $\two$ separate points of semilattices, the homomorphism $\delta:X/_{\Updownarrow}\to \two^{H}$, $\delta:x\mapsto(h(x))_{h\in H}$, is injective.  Since the semilattice $X/_{\Updownarrow}$ is  $\mathsf{T_{\!z}S}$-closed and $\mathsf{T_{\!z}S}$-discrete, the image $\delta[X/_{\Updownarrow}]$ is a closed discrete subsemilattice of the compact topological semilattice $\two^H$. Hence $X/_{\Updownarrow}$ is finite and so is the set $V\!E(X)$.
\end{proof}

Theorem~\ref{t:discrete} and Lemma~\ref{l:pd=>fin} imply the following lemma.

\begin{lemma} For any absolutely $\mathsf{T_{\!1}S}$-closed semigroup $X$, the sets  $X/_{\Updownarrow}$ and $V\!E(X)$ are finite.
\end{lemma}

\section{Absolutely $\mathsf{T_{\!1}S}$-closed semigroups}\label{s:T1}

In this section we establish some properties of absolutely $\mathsf{T_{\!1}S}$-closed semigroups and prove the following theorem that implies the characterization Theorem~\ref{t:main1} announced in the introduction.

\begin{theorem}\label{t:mainT1} For any semigroup $X$ we have implications $(1)\Ra(2)\Leftrightarrow(3)\Ra (4)\Ra(5)\Ra(6)$ of the following statements:
\begin{enumerate}
\item $X$ is finite;
\item $X$ is absolutely $\mathsf{T_{\!1}S}$-closed;
\item $X$ is projectively $\mathsf{T_{\!1}S}$-closed and projectively $\mathsf{T_{\!1}S}$-discrete;
\item $X$ is projectively $\mathsf{T_{\!z}S}$-closed and projectively $\mathsf{T_{\!z}S}$-discrete;
\item $X$ is projectively $\mathsf{T_{\!z}S}$-closed and $Z(X)\cap H(X)\cap\korin{\IN}{V\!E(X)}$ is finite;
\item $Z(X)$ is periodic and $Z(X)\cap \korin{\IN}{V\!E(X)}$ is finite.
\end{enumerate}
If $X$ is commutative, then the conditions $(1)$--$(6)$ are equivalent.
\end{theorem}

\begin{proof} The implication $(1)\Ra(2)$ is trivial, the equivalence $(2)\Leftrightarrow(3)$ was proved in Theorem~\ref{t:discrete}, and the implication $(3)\Ra(4)$ is trivial.
\smallskip

To prove that $(4)\Ra(5)$, assume that the semigroup $X$ is projectively $\mathsf{T_{\!z}S}$-closed and projectively $\mathsf{T_{\!z}S}$-discrete.
By Lemma~\ref{l:pd=>fin} the set $V\!E(X)$ is finite and by Theorem~\ref{t:Z}(1,2), the semigroup $Z(X)$ is periodic and group-finite. Then for every $e\in E(X)$ the intersection $Z(X)\cap H_e$ is either empty or  a finite subgroup of $Z(X)$. In both cases, the set $Z(X)\cap H_e$ is finite.  Then the set $$Z(X)\cap H(X)\cap \korin{\IN}{V\!E(X)}=\bigcup_{e\in V\!E(X)}(Z(X)\cap H_e)$$is finite, being the union of finitely many finite sets.
\smallskip

To prove that $(5)\Ra(6)$, assume that the semigroup $X$ is projectively $\mathsf{T_{\!z}S}$-closed and the set $Z(X)\cap H(X)\cap\korin{\IN}{V\!E(X)}$ is finite. By Theorem~\ref{t:Z}(1), the semigroup $Z(X)$ is periodic. By Theorem~\ref{t:ideal}, the set $Z(X)\cap\korin{\infty}{V\!E(X)}\setminus H(X)$ is finite and then the set $$Z(X)\cap \korin{\IN}{V\!E(X)}=\big(Z(X)\cap\korin{\IN}{V\!E(X)}\setminus H(X)\big)\cup\big(Z(X)\cap\korin{\IN}{V\!E(X)}\cap H(X)\big)$$is finite, too.
\smallskip

Now assuming that $X$ is commutative, we shall prove that $(6)\Ra(1)$.
So, assume that the semigroup $Z(X)$ is periodic and $Z(X)\cap\korin{\IN}{V\!E(X)}$ is finite. Being commutative, the semigroup $X$ is viable and hence $V\!E(X)=E(X)$. The periodicity of $Z(X)$ implies that $Z(X)=Z(X)\cap\korin{\IN}{E(X)}=Z(X)\cap\korin{\IN}{V\!E(X)}$ and hence the commutative semigroup $X=Z(X)$ is finite.
\end{proof}

\begin{corollary} If a $Z$-viable semigroup is absolutely $\mathsf{T_{\!1}S}$-closed, then its center $Z(X)$ is finite and hence $\mathsf{T_{\!1}S}$-closed.
\end{corollary}

\begin{proof} The $Z$-viability of $X$ yields $E(X)\cap Z(X)\subseteq V\!E(X)$. By Theorem~\ref{t:Z}(1), the semigroup $Z(X)$ is periodic and hence $Z(X)=Z(X)\cap\korin{\IN}{E(X)\cap Z(X)}\subseteq Z(X)\cap \korin{\IN}{V\!E(X)}$ is finite, by Theorem~\ref{t:mainT1}.
\end{proof}

\section{A sufficient condition of the absolute $\mathsf{T_{\!2}S}$-closedness}\label{s:if}

In this section we shall prove a sufficient condition of the absolute $\mathsf{T_{\!2}S}$-closedness. We shall use the following theorem, proved by Stepp in \cite[Theorem 9]{Stepp75}.

\begin{theorem}[Stepp]\label{t:E} Every chain-finite semilattice is absolutely $\mathsf{T_{\!2}S}$-closed.
\end{theorem}

A semigroup $X$ is called {\em $E$-commutative} if $xy=yx$ for any idempotents $x,y\in E(X)$.

\begin{lemma}\label{l:chain-closed2} Each chain-finite group-finite bounded Clifford+finite $E$-commutative semigroup $X$ is absolutely $\mathsf{T_{\!2}S}$-closed.
\end{lemma}

\begin{proof} To show that $X$ is absolutely $\mathsf{T_{\!2}S}$-closed, take any homomorphism $h:X\to Y$ to a Hausdorff topological semigroup $Y$. We should prove that the semigroup $h[X]$ is closed in $Y$. Replacing $Y$ by $\overline{h[X]}$, we can assume that $h[X]$ is dense in $Y$. Since $X$ is bounded, there exists $n\in\mathbb N$ such that $x^n\in E(X)$ and hence $x^{2n}=x^n$ for every $x\in X$. Taking into account that $h$ is a homomorphism, we conclude that $y^{2n}=y^n$ for all $y\in h[X]$. The closed subset $\{y\in Y:y^{2n}=y^n\}$ of $Y$ contains the dense set $h[X]$ and hence coincides with $Y$. Therefore, $y^{n}\in E(Y)$ for all $y\in Y$. It follows that the continuous map $\phi:Y\to E(Y)$, $\phi:y\mapsto y^n$, is well-defined.
Consider the function $\psi:X\to E(X)$, $\psi:x\mapsto x^n$, and observe that $h\circ \psi(x)=h(x^n)=(h(x))^n=\phi\circ h(x)$ for every $x\in X$.

Since $X$ is a chain-finite $E$-commutative semigroup, the set $E(X)$ is a chain-finite subsemilattice of $X$.  By Theorem~\ref{t:E}, the chain-finite semilattice $E(X)$ is absolutely $\mathsf{T_{\!2}S}$-closed and hence its image $h[E(X)]$ is closed in the Hausdorff topological semigroup $Y$.
The continuity of the map $\phi:Y\to E(Y)$, $\phi:y\mapsto y^n$, implies that
$$h[E(X)]\subseteq E(Y)= \phi[Y]=\phi[\overline{h[X]}]\subseteq\overline{\phi[h[X]]}=\overline{h[\psi[X]]}=\overline{h[E(X)]}=h[E(X)].$$ Hence $h[E(X)]=E(Y)=\phi[Y]$.
The choice of $n$ implies that $x=x^{n+1}$ for all $x\in H(X)$. Since $X$ is Clifford+finite, the set $F=X\setminus H(X)$ is finite. Then $Y=\overline{h[X]}=\overline{h[H(X)\cup F]}=\overline{h[H(X)]}\cup h[F]$. By the Hausdorff property of $Y$, the set $\{y\in Y:y=y^{n+1}\}\supseteq h[H(X)]$ is closed in $Y$ and contains the set $\overline{h[H(X)]}=\overline{h[X\setminus F]}\supseteq\overline{h[X]}\setminus h[F]=Y\setminus h[F]$. Then $y^{n+1}=y$ for any $y\in Y\setminus h[F]$.

Assuming that $h[X]$ is not closed in $Y$, take any point $y\in Y\setminus h[X]\subseteq Y\setminus h[F]$ and consider the idempotent $e=y^n=\phi(y)\in \phi[Y]=E(Y)=h[E(X)]$. Since the semilattice $E(X)$ is chain-finite, we can apply Theorem~\ref{t:mainP} and conclude that the semilattice $h[E(X)]=E(Y)$ is chain-finite and so is the subsemilattice $L=\{f\in E(Y):ef\ne e\}$ of $E(Y)$. By Theorem~\ref{t:E}, $L$ is closed in $E(Y)$. Then its complement ${\uparrow}e=\{f\in E(Y):ef=e\}$ is open in $E(Y)$ and its preimage $U=\phi^{-1}[{\uparrow}e]$ is an open neighborhood of $y$ in $Y$. Since $e\in h[E(X)]$ and the semilattice $E(X)$ is chain-finite, the nonempty subsemilattice $h^{-1}[{\uparrow}e]\cap E(X)$ has a unique minimal element $\mu\in h^{-1}(e)$. Since the semigroup $X$ is group-finite, the maximal subgroup $H_{\mu}$ is finite and so is the set $h[H_{\mu}\cup F]$. Since $y=y^{n+1}=ey\notin h[H_{\mu}\cup F]$, there exists a neighborhood $O_y\subseteq U$ of $y$ in $Y$ such that $O_y^{n+1}\subseteq U\setminus h[H_{\mu}\cup F]$. It implies that $eO_y\cap h[H_{\mu}\cup F]=\emptyset$, as $eO_y\subset O_y^{n+1}$. Since $\phi(ey)=(ey)^n=(y^{n+1})^n=(y^n)^{n+1}=e^{n+1}=e\in {\uparrow}e$, we can additionally assume that  $\phi[eO_y]\subseteq {\uparrow}e$.

Since
$$y\in Y\setminus h[X]=(\overline{h[H(X)]}\cup h[F])\setminus h[X]=\overline{h[H(X)]}\setminus h[X],$$ we can choose an element $x\in H(X)\cap h^{-1}[O_y]$ and observe that $h((\mu x)^n)=(h(\mu)h(x))^n=\phi(eh(x))\in\phi[eO_y]\subseteq {\uparrow}e$ and hence $(\mu x)^n\in E(X)\cap h^{-1}[{\uparrow}e]$. The minimality of $\mu$ in $E(X)\cap h^{-1}[{\uparrow}e]$ ensures that $\mu\le(\mu x)^n$. On the other hand, $\mu(\mu x)^n=(\mu x)^n$ implies that $(\mu x)^n\le\mu$ and hence $\mu=(\mu x)^n$ and $\mu x\in H_\mu\cup F$. On the other hand, $h(\mu x)=h(\mu)h(x)\in eO_y$ and hence $h(\mu x)\in h[H_{\mu}\cup F]\cap eO_y=\emptyset$. This contradiction shows that the set $h[X]$ is closed in $Y$.
\end{proof}

\section{Bounded sets in absolutely $\mathsf{T_{\!z}S}$-closed semigroups}\label{s:bounded}

In this section, given an absolutely $\mathsf{T_{\!z}S}$-closed semigroup $X$, we characterize subsets $A\subseteq V\!E(X)$ for which the set $Z(X)\cap\korin{\IN}{\!A}$ is bounded in $X$. We recall that a subset $B\subseteq X$ is {\em bounded} if $B\subseteq\korin{n}{E(X)}\defeq\{x\in X:x^n\in E(X)\}$ for some $n\in\IN$.

The following notion plays a crucial role in our subsequent results.

\begin{definition} A semigroup $X$ is defined to be {\em $A$-centrobounded} over a set $A\subseteq E(X)$ if there exists $n\in\IN$ such that for every $e\in A$ and $x,y\in \bigcap_{a\in A}\frac{H_a}a$ with $(xe)(ye)^{-1}\in H_e\cap Z(X)$ we have $((xe)(ye)^{-1})^n\in E(X)$.
\end{definition}

In the following theorem we endow the set $V\!E(X)$ with the natural partial order $\le$ considered in Section~\ref{s:order}. A subset $A\subseteq V\!E(X)$ is called an {\em antichain} if $x\not\le y$ for any distinct elements $x,y\in A$.

\begin{theorem}\label{t:bounded} Let $X$ be an absolutely $\mathsf{T_{\!z}S}$-closed semigroup. For a subset $A\subseteq V\!E(X)$ the following conditions are equivalent:
\begin{enumerate}
\item $Z(X)\cap \korin{\IN}{\!A}$ is bounded in $X$;
\item $Z(X)\cap H(X)\cap\korin{\IN}{\!A}$ is bounded in $X$;
\item $X$ is $B$-centrobounded over every countable  infinite antichain $B\subseteq A$.
\end{enumerate}
\end{theorem}

\begin{proof} Replacing the semigroup $X$ by its $1$-extension $X^1$, we lose no generality assuming that the semigroup $X$ contains a two-sided unit $1$. By Theorem~\ref{t:Z}(1,2), the semigroup $Z(X)$ is chain-finite, periodic, nonsingular, and group-finite.

The equivalence $(1)\Leftrightarrow(2)$ follows from  Theorem~\ref{t:ideal} and $(2)\Ra(3)$ is trivial. Indeed, by (2), there exists $n\in\IN$ such that $Z(X)\cap H(X)\cap\korin{\IN}{A}\subseteq \korin{n}{E(X)}$. We claim that the number $n$ witnesses that $X$ is $B$-centrobounded over any set $B\subseteq A$. Indeed, given any idempotent $e\in B$ and elements $x,y\in \bigcap_{b\in B}\frac{H_b}b$ with $(xe)(ye)^{-1}\in Z(X)$, by the periodicity of $Z(X)$ and the choice of $n$, we have $$(xe)(ye)^{-1}\in Z(X)\cap H_e\cap\korin{\IN}{e}\subseteq Z(X)\cap H(X)\cap\korin{\IN}{A}\subseteq \korin{n}{E(X)}$$ and hence $((xe)(ye)^{-1})^n\in E(X)$.

It remains to prove the implication $(3)\Ra(2)$. Let $\pi:\korin{\IN}{\!E(X)}\to E(X)$ be the map assigning to each element $x\in\korin{\IN}{\!E(X)}$ a unique idempotent $\pi(x)$ in the monogenic semigroup $x^\IN\defeq\{x^n:n\in\IN\}$. To derive a contradiction, assume that the condition (3) is satisfied but (2) does not.

\begin{claim}\label{cl:antichain} There exists a sequence $(z_k)_{k\in\w}$ in $Z(X)\cap H(X)\cap\korin{\IN}{\!A}$ such that
\begin{enumerate}
\item $\pi(z_k)\not\le \pi(z_n)$ for any distinct numbers $k,n\in\w$;
\item $z_k^i\ne z_k^j$ for any $k\in\w$ and distinct numbers $i,j\in\{1,\dots,2^k\}$.
\end{enumerate}
\end{claim}

\begin{proof} Since the set $Z(X)\cap H(X)\cap\korin{\IN}{\!A}$ is unbounded in $X$, for every $k\in\w$ there exists an element $g_k\in Z(X)\cap H(X)\cap \korin{\IN}{\!A}$ such that for any distinct positive numbers $i,j\le 2^k$ we have $g_k^i\ne g_k^j$. Let $[\w]^2$ be the family of two-element subsets of $\w$. Consider the function $\chi:[\w]^2\to\{0,1,2\}$ defined by
$$\chi(\{n,m\})=\begin{cases}0&\mbox{if $\pi(g_n)=\pi(g_m)$};\\
1&\mbox{if $\pi(g_n)<\pi(g_m)$ or $\pi(g_m)<\pi(g_n)$};\\
2&\mbox{otherwise}.
\end{cases}
$$
By the Ramsey Theorem 5 \cite{Ramsey}, there exists an infinite set $\Omega\subseteq \w$ such that $\chi[[\Omega]^2]=\{c\}$ for some $c\in\{0,1,2\}$. If $c=0$, then the set $\{\pi(g_n)\}_{n\in\Omega}$ contains a unique idempotent $u$ and hence the set $\{g_k\}_{k\in\Omega}\subseteq Z(X)\cap H_u$ is finite (since $Z(X)$ is periodic and group-finite). By the Pigeonhole Principle, for any $k>|Z(X)\cap H_u|$ there are two numbers $i<j\le k$ such that $g_k^i=g_k^j$, which contradicts the choice of $g_k$.  Therefore, $c\ne 0$. If $c=1$, then the set $\{\pi(g_k)\}_{k\in\Omega}$ is an infinite chain in $E(X)\cap Z(X)$ which is not possible as $Z(X)$ is chain-finite. Therefore, $c=2$ and hence $\{\pi(g_k)\}_{k\in\Omega}$ is an infinite antichain in $E(Z(X))$. Write the infinite set $\Omega$ as $\{n_k\}_{k\in\w}$ for some strictly increasing sequence $(n_k)_{k\in\w}$. For every $k\in\w$ put $z_k=g_{n_k}$ and observe that the sequence $(z_k)_{k\in\w}$ satisfies the conditions (1), (2) of Claim~\ref{cl:antichain}.
\end{proof}

Let $q:X\to X/_{\Updownarrow}$ be the quotient homomorphism of $X$ onto its semilattice reflection.

Let $(z_n)_{n\in\w}$ be the sequence from Claim~\ref{cl:antichain}. For every $n\in\w$ let $e_n=\pi(z_n)$. The inclusion $z_n\in Z(X)\cap \korin{\IN}{\!A}$ and the periodicity of $Z(X)$ imply that the idempotent $e_n=\pi(z_n)\in Z(X)\cap A\subseteq V\!E(X)$ is viable. Then the set $$\tfrac{H_{e_n}}{e_n}=\{x\in X:xe_n=e_nx\in H_{e_n}\}$$ is a prime coideal in $X$ and moreover, $\frac{H_{e_n}}{e_n}=q^{-1}[{\uparrow}q(e_n)]$, see Proposition 2.15 in \cite{BanE}.

Since the semigroup $X$ is absolutely $\mathsf{T_{\!z}S}$-closed, for the ideal $$\textstyle I\defeq X\setminus\bigcup_{n\in\w}\tfrac{H_{e_n}}{e_n}=X\setminus q^{-1}\big[\bigcup_{n\in\w}{\uparrow}q(e_n)\big]$$ in $X$, the quotient semigroup $X/I$ is absolutely $\mathsf{T_{\!z}S}$-closed.

For convenience, by $0$ we denote the element $I\in X/I$. The injectivity of the restriction $q{\restriction}_{V\!E(X)}$ implies that $e_ne_m\in I$ for any distinct $n,m\in\w$. This implies that the ideal $I$ is not empty and the element $0=I$ of the semigroup $X/I$ is well-defined.

Now we introduce a $0$-discrete Hausdorff semigroup topology $\tau$ on the semigroup $Y\defeq X/I$.

Fix any free ultrafilter $\mathcal{F}$ on $\mathbb{N}$.
Let
$$Q=\{y\in X/_{\Updownarrow}:\exists F\in\mathcal{F} \;\forall n\in F\;\; q(e_n)\leq y\}.$$
Note that the set $Q$ is nonempty, as $q(1)\in Q$ (we assumed that $X$ contains a unit exactly to omit the easier case $Q=\emptyset$).

For any $y_1,y_2\in Q$ there exist $F_1, F_2\in\mathcal{F}$ such that $q(e_n)\leq y_i$ for each $n\in F_i$, $i\in\{1,2\}$.
Then $q(e_n)\leq y_1y_2$ for each $n\in F_1\cap F_2\in \mathcal{F}$. It follows that $Q$ is a subsemilattice of $X/_{\Updownarrow}$. By Lemma~\ref{l:pc=>cf}, the semilattice $X/_{\Updownarrow}$ is chain-finite and so is its subsemilattice $Q$. Thus, the semilattice $Q$ contains the smallest element $s$. Since $s\in Q$, there exists a set $F_s\in\mathcal{F}$ such that $q(e_n)\leq s$ for all $n\in F_s$.  Consider the prime coideal
$$C\defeq \bigcap_{n\in F_s}\tfrac{H_{e_n}}{e_n}.$$

\begin{claim} ${\uparrow}s=\bigcap_{n\in F_s}{\uparrow}q(e_n)$ and hence $q^{-1}[{\uparrow}s]=C$.
\end{claim}

\begin{proof} The inclusion ${\uparrow}s\subseteq\bigcap_{n\in F_s}{\uparrow}q(e_n)$ follows from the choice of $F_s$. Now take any $y\in\bigcap_{n\in F_s}{\uparrow}q(e_n)$ and observe that $\{n\in\w:q(e_n)\le y\}\in \F$ and hence $y\in Q$ and $s\le y$ by the choice of $s$.
Then $$q^{-1}[{\uparrow}s]=q^{-1}[\bigcap_{n\in F_s}{\uparrow}q(e_n)]=\bigcap_{n\in F_s}q^{-1}[{\uparrow}q(e_n)]=\bigcap_{n\in F_s}\tfrac{H_{e_n}}{e_n}=C.$$
\end{proof}

To introduce the topology $\tau$ on $Y$, we need the following denotations.
For a real number $r$ by $\lfloor r\rfloor\defeq\max\{n\in\IZ:n\le r\}$ we denote the integer part of $r$. For each $k\in \mathbb{N}$ and $n\in F_s$ let $$A(n,k)=\bigcup_{2\le p\leq\lfloor 2^n/k\rfloor}z_n^pC\subseteq H_{e_n}.$$The definition of the set $A(n,k)$ implies that $A(n,k)=\emptyset$ if $2^n/k<2$.
For every $k\in\IN$ and $F\in\F$ consider the subset $$U_{k,F}\defeq\{0\}\cup\bigcup_{n\in F\cap F_s}A(n,k)$$
of $Y$.
On the semigroup $Y=X/I$ consider the topology $\tau$ generated by the base $$\big\{\{y\}:y\in Y\setminus \{0\}\big\}\cup\{U_{k,F}:k\in\IN,\;F\in\F\}.$$

\begin{claim} $(Y,\tau)$ is a Hausdorff zero-dimensional topological semigroup.
\end{claim}

\begin{proof} To see that the topology $\tau$ is Hausdorff, it suffices to show that for any $y\in Y\setminus\{0\}$ there exist $k\in\IN$ and $F\in\F$ such that $y\notin U_{k,F}$.
If $y\notin\bigcup_{n\in F_s}H_{e_n}$, then $y\notin U_{1,F_s}$. If $y\in H_{e_n}$ for some $n\in F_s$, then $y\notin U_{1,F_s\setminus\{n\}}$. Therefore, the topology $\tau$ is Hausdorff. Since $0$ is a unique non-isolated point of $(Y,\tau)$, the topology $\tau$ is zero-dimensional.

It remains to prove that $(Y,\tau)$ is a topological semigroup. Given points $y,y'\in Y$ and a neighborhood $O_{yy'}\in\tau$ of their product $yy'$ in $Y$, we need to find neighborhoods $O_y,O_{y'}\in\tau$ of $y,y'$, respectively, such that $O_yO_{y'}\subseteq O_{yy'}$. If $y,y'\in Y\setminus \{0\}$, then the neighborhoods $O_{y}=\{y\}$ and $O_{y'}=\{y'\}$ have the required property: $O_yO_{y'}=\{yy'\}\subseteq O_{yy'}$.

So, it remains to consider three cases:
\begin{enumerate}
\item[1)] $y\ne 0$ and $y'=0$;
\item[2)] $y=0$ and $y'\ne0$;
\item[3)] $y=0=y'$.
\end{enumerate}
In each of these cases, $yy'=0$, so we can find $F\in\F$ and $k\in\IN$ such that $F\subseteq F_s$ and $U_{k,F}\subseteq O_{yy'}$.
\smallskip

1) Assume that $y\ne 0$ and $y'=0$. If $y\in C$, then
\begin{multline*}
yU_{k,F}=\{0\}\cup\bigcup_{n\in F}yA(n,k)=\{0\}\cup\bigcup_{n\in F}\bigcup_{2\le p\le\lfloor 2^n/k\rfloor}yz_n^pC=\\ \{0\}\cup\bigcup_{n\in F}\bigcup_{2\le p\le\lfloor 2^n/k\rfloor}z_n^pyC\subseteq \{0\}\cup\bigcup_{n\in F}\bigcup_{2\le p\le\lfloor 2^n/k\rfloor}z_n^pC=U_{k,F}\subseteq O_{yy'}.
\end{multline*}
So, we can put $O_y=\{y\}$ and $O_{y'}=U_{k,F}$.

If $y\notin C$, then $q(y)\notin Q$ and hence the set $\{n\in\IN:q(e_n)\le q(y)\}$ does not belong to the ultrafilter $\F$ and then the set $G\defeq\{n\in F_s:q(e_n)\not\le q(y)\}$ belongs to the ultrafilter $\F$. For every $n\in G$ and $p\in\IN$ we have $q(yz_n^p)=q(y)q(e_n)\ne q(e_n)$, implying  $yz_n^p\in I$ and  $yU_{1,G}=\{0\}$. So we can put $O_y=\{y\}$ and $O_{y'}=U_{1,G}$.
\smallskip

2) The case $y=0$ and $y'\ne0$ can be treated by analogy with the preceding case.
\smallskip

3) If $y=0=y'$, then we can put $O_y=O_{y'}=U_{4k,F}$. Let us show that $O_yO_{y'}\subseteq U_{k,F}$. Indeed, take any $a,b\in U_{4k,F}$. If $0\in\{a,b\}$ or $a$ and $b$ do not belong to the same subgroup $H_{e_n}$, then $ab=0\in U_{k,F}$. Otherwise, there exists $n\in\IN$ such that $a,b\in A(n,4k)$ and hence $a\in z_n^mC$ and $b=z_n^tC$ for some numbers $m,t\in\{2,\dots,\lfloor 2^n/4k\rfloor\}$. Then $2\le \lfloor 2^n/4k\rfloor\le 2^n/4k$ and hence
$$2\le m+t\le 2\lfloor 2^n/4k\rfloor\le 2^n/2k=2^n/k-2^n/2k\le 2^n/k-4\le \lfloor 2^n/k\rfloor.$$
Taking into account that $z_n\in Z(X)$, we obtain $ab\in z_n^mCz_n^tC\subseteq z_n^{m+t}C\subseteq A(n,k)\subseteq U_{k,F}$.
\end{proof}

Note that $(Y,\tau)$, being a  continuous homomorphic image of the absolutely  $\mathsf{T_{\!z}S}$-closed semigroup $X$, is itself $\mathsf{T_{\!z}S}$-closed. But, as we will show further, this is not the case.

Let $z_{\F}$ be the ultrafilter on $Y$ generated by the base $\{z_F:F\in\F\}$ where $z_F\defeq\{z_n:n\in F\}$ for $F\in\F$. Note that for any $y\in Y$ the filter $y z_\F$ generated by the base $\{yG: G\in z_\F\}$ is an ultrafilter on $Y $. Also, since $\{z_n:n\in\IN\}\subseteq Z(X)$ we get that $yz_\F=z_\F y$, where $z_\F y$ is the ultrafilter generated by the base $\{Gy: G\in z_\F\}$.

\begin{claim} For any $y\in Y\setminus C$ the ultrafilter $y z_\F$ is the principal ultrafilter at $0$.
\end{claim}

\begin{proof}  The claim is obvious if $y=0$. So, assume that $y\in Y\setminus\{0\}=X\setminus I$. Since $y\in Y\setminus C$, the set $\{n\in\IN:q(e_n)\le q(y)\}$ does not belong to the ultrafilter $\F$. Then the set $F=\{n\in F_s:q(e_n)\not\le q(y)\}$ belongs to $\F$. Now observe that $yz_F\subseteq I$ and hence $yz_F=\{0\}$. Therefore, the ultrafilter $yz_\F$ is principal at $0$.
\end{proof}

\begin{claim}\label{cl:distinct} There exists $m\in\IN$ such that $U_{m,F_s}\notin y {x}_\F$ for every $y\in C$.
\end{claim}

\begin{proof} Since $X$ is $\{e_n\}_{n\in F_s}$-centrobounded, there exists $m\in\IN$ such that for every $n\in F_s$ and $x,y\in C=\bigcap_{k\in F_s}\frac{H_{e_k}}{e_k}$ with $(xe_n)(ye_n)^{-1}\in Z(X)$ we have $\big((xe_n)(ye_n)^{-1}\big)^m\in E(X)$.

We claim that for every $y\in C$, the set $U_{m,F_s}$ does not belong to the ultrafilter $yz_\F$. In the opposite case, the set $U_{m,F_s}$ has non-empty intersection with the set $yz_{F_s}$. Then there exists $n\in F_s$ such that $yz_n\in U_{m,F_s}$ and hence $z_ny=yz_n=z_n^pc$ for some $c\in C$ and $p\in\IN$ with $2\le p\le\lfloor 2^n/m\rfloor$.  It follows from $c,y\in C\subseteq \frac{H_{e_n}}{e_n}$ that $ce_n,ye_n\in H_{e_n}$. Then the equality $z_ny=z_n^pc$ implies $z_nye_n=z_n^pce_n\in H_{e_n}$ and hence $(ce_n)(ye_n)^{-1}=z_n^{-(p-1)}\in Z(X)\cap H_{e_n}$.
Now the choice of $m$ ensures that $((ce_n)(ye_n)^{-1})^m=e_n$. Then $$z_n^m=(z_n^p(ce_n)(ye_n)^{-1})^m=z_n^{pm}((ce_n)(ye_n)^{-1})^m=z_n^{pm}e_n=z_n^{pm},$$
which contradicts the choice of the point $z_n$ in Claim~\ref{cl:antichain}(2) as $m<pm\le\lfloor 2^n/m\rfloor m\le 2^n$.
\end{proof}

Let $T=Y\cup \{yz_\F:y\in C\}$.
Extend the semigroup operation from $Y$ to the set $T$ by the formula:
$$ab=
\begin{cases}
ab&\hbox{ if $a,b\in Y$};\\
ayz_\F&\hbox{ if  $a\in C$ and $b=yz_\F$ for some $y\in C$};\\
ybz_\F&\hbox{ if  $a=yz_\F$ for some $y\in C$ and $b\in C$};\\
0&\mbox{ in all other cases}.
\end{cases}
$$
Let $\theta$ be the topology on the semigroup $T$ which satisfies the following conditions:
\begin{itemize}
\item $(Y,\tau)$ is an open subspace of $(T,\theta)$;
\item if $yz_\F\in U\in \theta$ for some $y\in C$, then there exists $F\in\F$ such that $yz_F\subseteq U$.
\end{itemize}

\begin{claim}  The topology $\theta$ on $T$ is Hausdorff and zero-dimensional.
\end{claim}

\begin{proof} First we show that the topological space $(T,\theta)$ is zero-dimensional. Given an open set $U\in\theta$ and a point $u\in U$, we need to find a clopen set $V$ in $(T,\theta)$ such that $u\in V\subseteq U$. We consider three possible cases.
\smallskip

1. If $u\in Y\setminus \{0\}$, then $u$ is an isolated point of $Y$ and $T$ and we can put $V=\{u\}$. The definition of the topology $\theta$ ensures that $V=\{u\}$ is a clopen neighborhood of $u$ in $(T,\theta)$.
\smallskip

2. If $u=0$, then we can apply Claim~\ref{cl:distinct} and find $m\in\IN$ and $F\in\F$ such that $U_{m,F}\subseteq U$ and  $U_{m,F}\notin yz_\F$ for all $y\in C$. The definition of the topology $\theta$ ensures that $V=U_{m,F}$ is a clopen neighborhood of $u=0$ in $(T,\theta)$.
\smallskip

3. If $u=yz_\F$ for some $y\in C$, then by the definition of the topology $\theta$, there exists $F\in\F$ such that $F\subseteq F_s$ and $yx_F\subseteq U$. Moreover, by Claim~\ref{cl:distinct}, we can assume that $yx_F\cap U_{m,F_s}=\emptyset$ for some $m\in\IN$. By the definition of the topology $\theta$, the set $V\defeq\{yz_\F\}\cup yx_F\subseteq U$ is a neighborhood of $u=yz_\F$ in $(T,\theta)$. It remains to show that the set $V$ is closed in $(T,\theta)$. Given any $t\in T\setminus V$, we should find a neighborhood $O_t\in\theta$ of $t$ such that $O_t\cap V=\emptyset$. If $t\in Y\setminus\{0\}$, then the neighborhood $O_t\defeq \{t\}\in\theta$ of $t$ is disjoint with $V$. If $t=0$, then the neighborhood $O_t\defeq U_{m,F_s}$  of $t=0$ is disjoint with $V$. Finally assume that $t=c z_\F$ for some $c\in C$. Consider the set $E=\{n\in F:yz_n=c z_n\}$. If $E\in\F$, then $u=yz_\F=c z_\F=t$, which contradicts the choice of $t\notin V$. Therefore, $E\notin\F$ and the set $G\defeq F\setminus E=\{n\in F:yz_n\ne c z_n\}$ belongs to the ultrafilter $\F$. Then $O_t=\{t\}\cup c z_{G}$ is a neighborhood of $t=c z_\F$ such that $O_t\cap V=\emptyset$, witnessing that the set $V$ is clopen.

Therefore the topology $\theta$ is zero-dimensional and being $T_1$, it is Hausdorff.
\end{proof}

To check the continuity of the semigroup operation in $(T,\theta)$, take any elements $a,b\in T$ and choose any neighborhood $O_{ab}\in\theta$ of their product $ab$. We should find neighborhoods $O_a,O_b\in\theta$ of $a,b$ such that $O_aO_b\subseteq O_{ab}$.
If $a,b\in Y$, then such neighborhoods exist by the continuity of the semigroup operation in the topological semigroup $(Y,\tau)$.

So, it remains to consider three cases:
\smallskip

1. $a\in Y$ and $b=yz_\F$ for some $y\in C$. This case has three subcases.
\smallskip

1a) $a\in C$. In this case there exists a set $F\in\F$ such that $ayx_F\subseteq O_{ab}$, and then the neighborhoods $O_a=\{a\}$ and $O_b=yx_F\cup\{yx_{\F}\}$ have the required property   $O_aO_b=ayx_F\cup\{ayx_{\F}\}\subseteq O_{ab}$.
\smallskip

1b) $a\in Y\setminus (\{0\}\cup C)=X\setminus(I\cup C)$.
Since $a\notin C$, the set $\{n\in \w:q(e_n)\le q(a)\}$ does not belong to the ultrafilter $\F$ and hence the set $F\defeq\{n\in F_{s}:q(e_n)\not\le  q(a)\}$ belongs to $\F$. Then for the neighborhoods $O_a=\{a\}$ and $O_b=yz_F\cup\{yz_\F\}$ we have $O_aO_b=\{0\}=\{ab\}\subseteq O_{ab}$.
\smallskip

1c) $a=0$. In this case $ab=0$ and  we can find $k\in\IN$ and $F\in\F$ such that $U_{k,F}\subseteq O_{ab}$. We claim that the neighborhoods $O_a=U_{2k,F}$ and $O_b=yz_{F}\cup\{yz_{\F}\}$ satisfy $O_aO_b\subseteq O_{ab}$. Given any elements $a'\in O_a$ and $b'\in O_b$, we should check that $a'b'\in O_{ab}$. This is clear if $a'=0$. If $a'\ne0$, then $a'=z_n^py'$ for some $n\in F$, $2\le p\le \lfloor 2^n/2k\rfloor$ and $y'\in C$. If $b'=yz_{\F}$, then $a'b'=0\in O_{ab}$. So, assume that $b'=yz_{n'}$ for some $n'\in F$. If $n\ne n'$, then $a'b'=z_n^py'yz_{n'}\in I$ and hence $a'b'=0=ab\in O_{ab}$. Suppose that $n=n'$. Then $a'b'=z_n^py'yz_n=z_n^{p+1}y'y$ and $$2\le p+1\le \lfloor 2^n/2k\rfloor+1\le 2^n/2k+1=2^n/k-2^n/2k+1\le 2^n/k-2+1\le\lfloor 2^n/k\rfloor$$ as $2\le p\le \lfloor 2^n/2k\rfloor\le 2^n/2k$.
\smallskip

2. $a=yz_\F$ for some $y\in C$ and $b\in Y$. This case can be considered by analogy with the preceding case.
\smallskip

3. $a=yz_\F$ and $b=y'z_\F$ for some $y,y'\in C$. In this case $ab=0$ and we can find $k\in\IN$ and $F\in\F$ such that $U_{k,F}\subseteq O_{ab}$. Let $E=\{n\in F\cap F_s:2\le \lfloor 2^n/k\rfloor\}$. Then the neighborhoods $O_a=yz_E\cup\{yz_{\F}\}$ and $O_b=y'z_E\cup\{y'z_\F\}$ have the required property: $O_aO_b\subseteq U_{k,F}\subseteq O_{ab}$.
\smallskip

Observe that the continuity of the binary operation in $(T,\theta)$ implies that it is associative, as $Y$ is a dense subsemigroup of $T$.
Thus, $(T,\theta)$ is a Hausdorff zero-dimensional topological semigroup which contains $(Y,\tau)$ as a non-closed subsemigroup. But this contradicts the absolute $\mathsf{T_{\!z}S}$-closedness of $(Y,\tau)$ and $X$. The obtained contradiction completes the proof of the implication $(3)\Ra(2)$ and also the proof of Theorem~\ref{t:bounded}.
\end{proof}

\section{Some sufficient conditions of centroboundedness}\label{s:Acb}

In this section we shall find some sufficient conditions of centroboundedness, which will be combined with Theorem~\ref{t:bounded} in order to obtain the boundedness of certain sets in absolutely $\mathsf{T_{\!z}S}$-closed semigroups.

The following theorem is the main result of this section.

\begin{theorem}\label{t:=>Acb} Let $X$ be a semigroup, $A\subseteq  V\!E(X)$ be a  countable set of viable idempotents. Consider the prime coideal $C\defeq\bigcap_{e\in A}\frac{H_e}e$ in $X$, the group $H\defeq\prod_{e\in A}H_e$, and the homomorphism $h:C\to H$, $h:x\mapsto(xe)_{e\in A}$. The semigroup $X$ is $A$-centrobounded if one of the following conditions is satisfied:
\begin{enumerate}
\item $X$ is projectively $\mathsf{T_{\!z}S}$-closed and for every $e\in A$ the subsemigroup $Ce$ of $H_e$ is commutative;
\item the semigroup $h[C]$ is polyfinite;
\item $X$ is projectively $\mathsf{T_{\!z}S}$-closed and $h[C]$ is countable;
\item $X$ is absolutely $\mathsf{T_{\!z}S}$-closed, $|h[C]|\le\cov(\M)$, and for every $e\in A$ the subsemigroup $Ce\subseteq H_e$ is countable.
\item $X$ is absolutely $\mathsf{T_{\!z}S}$-closed, $|h[C]|\le\mathfrak c$ and for every $e\in A$ the subsemigroup $Ce$ of $H_e$ is bounded.
\end{enumerate}
\end{theorem}

\begin{proof}  By our assumption, the set $A$ is countable and hence admits an injective function $\lambda:A\to\w$. Consider the group $H\defeq \prod_{e\in A}H_e$ endowed with the Tychonoff product topology of discrete topological groups $H_e$, $e\in A$. This topology is generated by the complete invariant metric $\rho:H\times H\to\IR$ defined by
$$\rho((x_a)_{a\in A},(y_a)_{a\in A})=\max(\{0\}\cup\{\tfrac1{2^{\lambda(a)}}:a\in A,\;x_a\ne y_a\}).$$
For every $e\in A$, consider the homomorphism $$\textstyle h_e:C\to H_e,\quad h:x\mapsto xe=ex.$$The homomorphisms $h_e$ compose the homomorphism $$\textstyle h:C\to H,\quad h:x\mapsto (h_e(x))_{e\in A}=(xe)_{e\in A}.$$ For every $a\in A$, let $\pr_a:H\to H_a$, $\pr_a:(x_e)_{e\in A}\mapsto x_a$, be the $a$th coordinate projection. The definition of the homomorphism $h$ implies that $\pr_a\circ h=h_a$ for every $a\in A$.
\smallskip

1. Assume that $X$ is projectively $\mathsf{T_{\!z}S}$-closed and for every $e\in A$ the subsemigroup $Ce$ of $H_e$ is commutative. Then the semigroup $h[C]\subseteq\prod_{e\in A}Ce$ is commutative. By Lemma~\ref{l:pci}, the prime coideal $C$ of  $X$ is projectively $\mathsf{T_{\!z}S}$-closed and so is its homomorphic image $h[C]$. By Theorem~\ref{t:Z}(1), the $\mathsf{T_{\!z}S}$-closed commutative semigroup $h[C]\subseteq H$ is periodic and hence is a subgroup of the group $H$. By Theorem~\ref{t:C-closed}, the $\mathsf{T_{\!z}S}$-closed commutative group $h[C]$ is bounded. Then there exists $n\in\IN$ such that $z^n\in E(H)$ for every $z\in h[C]$. Consequently, for every $x,y\in C$ and $e\in A$ we have that $((xe)(ye)^{-1})^n=\pr_e((h(x)h(y)^{-1})^n)=e$, witnessing that $X$ is $A$-centrobounded.
\smallskip

2. Assume that the semigroup $h[C]$ is polyfinite. Then there exist $n\in\IN$ and a finite set $F\subseteq h[C]$ such that for any $x,y\in h[C]$ there exists a semigroup polynomial $f:h[C]\to h[C]$ of degree $\le n$ such that $\{f(x),f(y)\}\subseteq F$. To show that $X$ is $A$-centrobounded, it suffices to check that for any $e\in A$ and $x,y\in C$ with $z\defeq (xe)(ye)^{-1}\in Z(X)$ we have $z^{m}=e$ where $m=(|F|n)^2!$.

For every $k\in\IN$ the assumption $z=(xe)(ye)^{-1}\in H_e\cap Z(X)$ implies $xe=zye$ and hence $(xe)^k=(zye)^k=z^k(ye)^k$. By the choice of $n$ and $F$, there exists a semigroup polynomial $f_k:h[C]\to h[C]$ of degree $\le n$ such that $\{f_k(h(x^k)),f_k(h(y^k))\}\subseteq F$. Find elements $a_0,\dots,a_{\deg(f_k)}\in h[C]$ such that $f_k(v)=a_0va_1v\cdots va_{\deg(f_k)}$ for all $v\in h[C]$. For every $i\in\{0,\dots,\deg(f_k)\}$, consider the element $\check a_i=\pr_e(a_i)$ of the group $H_e$. Let $\check f_k:H_e\to H_e$ be the semigroup polynomial defined by $\check f_k(v)=\check a_0v\check a_1v\dots v\check a_{\deg(f_k)}$ for $v\in H_e$. It is easy to see that $\pr_e\circ f_k=(\check f_k\circ \pr_e){\restriction}_{h[C]}$ and hence $\pr_e\circ f_k\circ h=\check f_k\circ\pr_e\circ h=\check f_k\circ h_e$. Then
$$\{\pr_e(f_k(h(x^k))),\pr_e(f_k(h(y^k)))\}=\{\check f_k(h_e(x^k)),\check f_k(h_e(y^k))\}=\{\check f_k(x^ke),\check f_k(y^ke)\}$$ and hence $\{\check f_k(x^ke),\check f_k(y^ke)\}\subseteq\pr_e[F]$. By the Pigeonhole Principle, there exists a triple $(a,b,d)\in \pr_e[F]\times \pr_e[F]\times\{1,\dots,n\}$ and two positive numbers $i<j\le 1+n|\pr_e[F]|^2$ such that $$(\check f_i(x^ie),\check f_i(y^ie),\deg(\check f_i))=(a,b,d)=(\check f_j(x^je),\check f_j(y^je),\deg(\check f_j)).$$
It follows from $x^ie=z^iy^ie$ and $z\in Z(X)$ that
$$a=\check f_i(x^ie)=\check f_{i}(z^iy^ie)=z^{i\deg(\check f_i)}\check f_i(y^ie)=z^{id}b.$$
By analogy we can prove that $a=z^{jd}b$. Since $H_e$ is a group, the equality $z^{id}b=z^{jd}b$ implies $z^{id}=z^{jd}$ and $z^{(j-i)d}=e$. Since $(j-i)d\le |F|^2 n^2$ divides $m=(|F|n)^2!$, $z^m=e$.
\smallskip

3. Assume that the semigroup $X$ is projectively $\mathsf{T_{\!z}S}$-closed and the set $h[C]$ is countable. By Lemma~\ref{l:pci}, the prime coideal  $C$ in $X$ is projectively $\mathsf{T_{\!z}S}$-closed and so is its homomorphic image $h[C]$. Being $\mathsf{T_{\!z}S}$-closed, the semigroup $h[C]$ is zero-closed, see Lemma~\ref{l:c=>zc}. By Theorem~\ref{t:polyb}(1), the zero-closed countable semigroup $h[C]$ is polybounded and by Lemma~\ref{l:polyb=>polyf}, $h[C]$ is polyfinite. By the preceding statement, the semigroup $X$ is $A$-centrobounded.
\smallskip

4. Assume that the semigroup $X$ is absolutely $\mathsf{T_{\!z}S}$-closed, $|h[C]|\le\cov(\M)$, and for every $e\in A$ the subsemigroup $Ce\subseteq H_e$ is countable. Since $h[C]\subseteq \prod_{e\in A}Ce\subseteq H$, the subsemigroup $h[C]$ of the metric group $(H,\rho)$ is separable. By Lemma~\ref{l:pci}, the prime coideal $C$ is  absolutely $\mathsf{T_{\!z}S}$-closed and so is its homomorphic image $h[C]$. By the absolute $\mathsf{T_{\!z}S}$-closedness of $h[C]$, the semigroup $h[C]$ is zero-closed and also $h[C]$ is closed in the zero-dimensional topological group $H$. Then the metric $\rho{\restriction}_{h[C]\times h[C]}$ is complete and hence $h[C]$ is a Polish space.

We claim that $h[C]$ is polybounded. If $|h[C]|<\mathfrak c$, then the Polish space  $h[C]$ is countable, see \cite[6.5]{Ke}.
By Theorem~\ref{t:polyb}(1), the countable zero-closed semigroup $h[C]$ is polybounded. If $|h[C]|=\mathfrak c$, then the inequality $|h[C]|\le\cov(\M)$ implies $\cov(\M)=\mathfrak c$ and then $h[C]$ is polybounded by Theorem~\ref{t:polyb}(2). So, in both cases, the semigroup $h[C]$ is polybounded. By Lemma~\ref{l:polyb=>polyf}, $h[C]$ is polyfinite. By the second statement of this theorem, the semigroup $X$ is $A$-centrobounded.
\smallskip

5. Assume that $|h[C]|\le\mathfrak c$ and for every $e\in A$ the semigroup $Ce\subseteq H_e$ of $X$ is bounded. For a bounded  subset $B\subseteq X$, let
$$\exp(B)\defeq\min\{n\in\IN:B\subseteq\korin{n}{\!E(X)}\}.$$ For a finite subset $F\subseteq \IN$ let
$$\lcm(F)\defeq\min\{n\in\IN:\forall x\in F\;\exists k\in\IN\;\;(n=xk)\}$$be the {\em least common multiple} of numbers in the set $F$.

To derive a contradiction, assume that the semigroup $X$ is not $A$-centrobounded. Writing down the negation of the $A$-centroboundedness, we obtain sequences $(x^+_n)_{n\in\w}$, $(x^-_n)_{n\in\w}$, $(z_n)_{n\in\w}$ and $(e_n)_{n\in\w}$ such that for every $n\in\w$ the following conditions are satisfied:
\begin{enumerate}
\item[(i)] $x^+_n,x^-_n\in C$, $e_n=Z(X)\cap A$, and $z_n=(x^+_ne_n)(x^-_ne_n)^{-1}\in Z(X)\cap H_{e_n}$;
\item[(ii)] $z_n^i\ne e_n$ for every $1\le i\le n\mu_n$ where $\mu_n\defeq\lcm\{\exp(Ce_k):k<n\}$.
\end{enumerate}
Consider the group $H'=\prod_{n\in\w}H_{e_n}$, and the homomorphism $h':C\to H'$, $h':x\mapsto (xe_n)_{n\in\w}$. Observe that $h'=\pr'\circ h$, where  $\pr':H\to H'$, $\pr':(x_a)_{a\in A}\mapsto (x_{e_n})_{n\in\w}$, is the projection. For every $n\in\w$, let $\pr'_n:H'\to H_{e_n}$, $\pr'_n:(x_k)_{k\in\w}\mapsto x_n$, be the $n$-th coordinate projection. Endow the group $H'$ with the complete invariant metric
$$\rho':H'\times H'\to\IR,\quad \rho'((x_n)_{n\in\w},(y_n)_{n\in\w})\mapsto(\{0\}\cup \{\tfrac1{2^n}:n\in\w,\;x_n\ne y_n\}).$$

By Lemma~\ref{l:pci}, the prime coideal $C$ of the absolutely $\mathsf{T_{\!z}S}$-closed semigroup $X$ is absolutely $\mathsf{T_{\!z}S}$-closed and so is its homomorphic image $h'[C]\subseteq H'$. Since the Tychonoff product topology on the group $H'=\prod_{n\in\w}H_{e_n}$ is zero-dimensional, the absolutely $\mathsf{T_{\!z}S}$-closed subsemigroup $h'[C]$ of $H'$ is closed in $H'$. Being absolutely $\mathsf{T_{\!z}S}$-closed, the semigroup $h'[C]$ is zero-closed by Lemma~\ref{l:c=>zc}. By Theorem~\ref{t:polyb}(1), the semigroup $h'[C]$ is $\kappa$-polybounded for some $\kappa<|h'[C]|=|\pr'[h[C]|\le|h[C]|\le\mathfrak c$. Then $h'[C]=\bigcup_{\alpha\in\kappa}f_\alpha^{-1}(b_\alpha)$ for some elements $b_\alpha\in h'[C]$ and some semigroup polynomials $f_\alpha:h'[C]\to h'[C]$.

Recall that $\mu_n$ is the least common multiple of the numbers $\exp(C{e_k})$, $k<n$, and observe that for every $x\in h'[C]\subseteq \prod_{n\in\w}Ce_n$, its inverse $x^{-1}$ in $H'$ is the limit of the sequence $(x^{1+\mu_n})_{n\in\w}$, which implies that $x^{-1}\in\overline{h'[C]}=h'[C]$ and means that $h'[C]$ is a subgroup of $H'$.

For every $n\in\w$ consider the element $$x^\pm_n=(h'(x^+_n)h'(x^-_n)^{-1})^{\mu_n}\in h'[C]\subseteq H'.$$

Observe that for every $k<n$ we have
\begin{multline*}
\pr'_k(x^\pm_n)=\pr'_k\big((h'(x^+_n)h'(x^-_n)^{-1})^{\mu_n}\big)=((x^+_ne_k)(x^-_ne_k)^{-1})^{\mu_n}=\\ \big(((x_n^+e_k)(x^-_ne_k)^{-1})^{\exp(C{e_k})}\big)^{\mu_n/\exp(C{e_k})}=
e_k^{\mu_n/\exp(C{e_k})}=e_k,
\end{multline*}
which means that the sequence $(x^\pm_n)_{n\in\w}$ converges to the identity element $e'=(e_k)_{k\in\w}$ of the topological group $H'=\prod_{k\in\w}H_{e_k}$.

Let $2=\{0,1\}$ and $2^{<\w}\defeq\bigcup_{n\in\w}2^n$. Define the family $(p_s)_{s\in 2^{<\w}}$ of elements of $H'$ by the recursive formula:
$$\mbox{$p_{\emptyset}=e'$ \ and \ $p_{s\hat{\;}0}=p_s$, $p_{s\hat{\;}1}=p_s\cdot x^\pm_n$ for every $n\in\w$ and $s\in 2^n$.}$$

The definition of the ultrametric $\rho'$ on $H'$ and the convergence of $x^\pm_n\to e'$ imply that for every $s\in 2^\omega$ the sequence $(p_{s{\restriction}_n})_{n\in\w}$ is Cauchy in the metric space $(H',\rho')$ and hence it converges to some element $p_s\in h'[C]\subseteq H'$. Since $\{p_s:s\in 2^\w\}\subseteq h'[C]\subseteq\bigcup_{\alpha\in\kappa}f_\alpha^{-1}(b_\alpha)$ and $\kappa<\mathfrak c$, there exists $\alpha\in\kappa$ such that the set $2^\w_\alpha\defeq\{s\in 2^\w:p_s\in f_\alpha^{-1}(b_\alpha)\}$ is uncountable and hence contains two distinct sequences $s,s'\in 2^\w_\alpha$ such that $s{\restriction}_{\deg(f_\alpha)}=s'{\restriction}_{\deg(f_\alpha)}$.
Let $m\in\IN$ be the smallest number such that $s(m)\ne s'(m)$. Then $m\ge\deg(f_\alpha)\ge 1$ and $s{\restriction}_m=s'{\restriction}_m$ by the minimality of $m$. Let $t=s{\restriction}_m=s'{\restriction}_m$ and observe that $\{p_{s{\restriction}(m+1)},p_{s'{\restriction}(m+1)}\}=\{p_{t\hat{\;}0},p_{t\hat{\;}1}\}$. We lose no generality assuming that $p_{s{\restriction}(m+1)}=p_{t\hat{\;}0}=p_{t}$ and
$p_{s'{\restriction}(m+1)}=p_{t\hat{\;}1}=p_tx_m^\pm$.
It follows from $f_\alpha(p_s)=b_\alpha=f_\alpha(p_{s'})$ that
$\pr'_m(f_\alpha(p_s))=\pr'_m(f_\alpha(p_{s'}))$.
Find elements $a_0,a_1,\dots, a_{\deg(f_\alpha)}\in h'[C]$ such that $f_\alpha(x)=a_0xa_1x\cdots xa_{\deg(f_\alpha)}$ for all $x\in h'[C]$.

For every $i\in\{0,\dots,\deg(f_\alpha)\}$ let $\check a_i=\pr'_m(a_i)$. Let $\check f_\alpha:H_{e_m}\to H_{e_m}$ be the semigroup polynomial defined by
$\check f_\alpha(x)=\check a_0x\check a_1x\cdots x\check a_{\deg(f_\alpha)}$ for $x\in H_{e_m}$. It is clear that $\pr'_m\circ f_\alpha=\check f_\alpha\circ\pr'_m$.

It follows from $\pr'_m(x_m^\pm)=z_m^{\mu_m}\in Z(X)$ and $\pr'_m(x_k^\pm)=e_m$ for all $m<k$ that
\begin{multline*}
\check f_\alpha(\pr'_m(p_t))=\check f_\alpha(\pr'_m(p_s))=\pr'_m(f_\alpha(p_s))=\pr'_m(f_{\alpha}(p_{s'}))=\check f_\alpha(\pr'_m(p_{s'}))=\check f_\alpha(\pr'_m(p_{t\hat{\;}1}))=\\
\check f_\alpha(\pr'_m(p_t)\pr'_m(x_m^\pm))=\check f_\alpha(\pr'_m(p_t)z_m^{\mu_m})=\check f_\alpha(\pr'_m(p_t))z_m^{\mu_m\deg{f_\alpha}}
\end{multline*}
and hence
$e_m=z_m^{\mu_m\deg{f_\alpha}}$, which contradicts the choice of $z_m$.
\end{proof}

\begin{corollary}\label{c:subsem} Let $X$ be an absolutely $\mathsf{T_{\!z}S}$-closed semigroup and $A\subseteq V\!E(X)$. Assume that for any infinite countable antichain $B\subseteq A$, the coideal $C\defeq\bigcap_{e\in B}\frac{H_e}e$ and the homomorphism $h:C\to H\defeq\prod_{b\in B}H_b$, $h:x\mapsto (xe)_{e\in B} $, one of the following conditions is satisfied:
\begin{enumerate}
\item for every $e\in B$ the subsemigroup $Ce$ of $H_e$ is commutative;
\item the semigroup $h[C]$ is polyfinite;
\item $h[C]$ is countable;
\item $|h[C]|\le\cov(\M)$, and for every $e\in A$ the subsemigroup $Ce\subseteq H_e$ is countable.
\item $|h[C]|\le\mathfrak c$ and for every $e\in A$ the subsemigroup $Ce$ of $H_e$ is bounded.
\end{enumerate}
Then the set $Z(X)\cap\korin{\IN}{\!A}$ is bounded, and every subsemigroup of $S\subseteq Z(X)\cap\korin{\IN}{\!A}$ of $X$ is  absolutely $\mathsf{T_{\!2}S}$-closed.
\end{corollary}

\begin{proof} By Theorems~\ref{t:bounded} and \ref{t:=>Acb}, the set $Z(X)\cap\korin{\IN}{\!A}$ is bounded. Now let $S\subseteq Z(X)\cap \korin{\infty}{\!A}$ be any subsemigroup of $X$.  By Theorem~\ref{t:Z}, the semigroup $S\subseteq Z(X)$ is chain-finite, group-finite, periodic and nonsingular. The periodicity of $S$ implies that $H(S)=S\cap H(X)$ and hence $S\setminus H(S)\subseteq Z(X)\cap\korin{\IN}{\!A}\setminus H(X)$. By Theorem~\ref{t:ideal}, the set $Z(X)\cap\korin{\infty}{\!A}\setminus H(X)\supseteq S\setminus H(S)$ is finite, which implies that the semigroup $S$ is Clifford+finite. By Lemma~\ref{l:chain-closed2}, the chain-finite group-finite bounded Clifford+finite commutative semigroup $S$ is absolutely $\mathsf{T_{\!2}S}$-closed.
\end{proof}

\begin{corollary}\label{c:IZ} If a semigroup $X$ is absolutely $\mathsf{T_{\!z}S}$-closed, then its ideal center $I\!Z(X)$ is bounded and absolutely $\mathsf{T_{\!2}S}$-closed.
\end{corollary}

\begin{proof} For every $e\in E(I\!Z(X))$ we have $H_e=H_ee\subseteq X\cdot I\!Z(X)\subseteq Z(X)$, which implies that the maximal subgroup $H_e$ is commutative.
By Theorem~\ref{t:Z}, the semigroup $Z(X)$ is chain-finite, group-finite, periodic, and nonsingular. By Lemma~\ref{l:viable}, $E(I\!Z(X))\subseteq V\!E(X)$ and by the periodicity of $I\!Z(X)$, we obtain $I\!Z(X)=I\!Z(X)\cap\korin{\IN}{E(I\!Z(X))}$. By Corollary~\ref{c:subsem}(1), the set $Z(X)\cap \korin{\IN}{E(I\!Z(X))}$ is bounded in $X$ and so is its subset $I\!Z(X)$.
By the periodicity, $H(I\!Z(X))=H(X)\cap I\!Z(X)$. By Theorem~\ref{t:ideal}, the set $$I\!Z(X)\setminus H(I\!Z(X))=I\!Z(X)\setminus H(X)= I\!Z(X)\cap\korin{\IN}{E(I\!Z(X))}\setminus H(X)\subseteq Z(X)\cap\korin{\IN}{V\!E(X)}\setminus H(X)$$ is finite, which means that the semigroup $I\!Z(X)$ is Clifford+finite.

Therefore, the commutative semigroup $I\!Z(X)$ is chain-finite, group-finite, bounded, and Clifford+finite. By Lemma~\ref{l:chain-closed2}, $I\!Z(X)$ is absolutely $\mathsf{T_{\!2}S}$-closed.
\end{proof}

\section{Proof of Theorem~\ref{t:ac=>ideal}}\label{s:ideal}

Let $\C$ be class of topological semigroups such that either $\C=\mathsf{T_{\!1}S}$ or $\mathsf{T_{\!2}S}\subseteq \C\subseteq \mathsf{T_{\!2}S}$.

Given an absolutely $\C$-closed semigroup, we should prove that the ideal center $I\!Z(X)$ of $X$ is absolutely $\C$-closed. By Lemma~\ref{l:viable}, $E(I\!Z(X))\subseteq V\!E(X)$. Since $\mathsf{T_{\!z}S}\subseteq\C$, the semigroup $X$ is absolutely $\mathsf{T_{\!z}S}$-closed. By Theorem~\ref{t:Z}, the semigroup $Z(X)$ is periodic and hence $$I\!Z(X)= I\!Z(X)\cap\korin{\IN}{E(I\!Z(X))}\subseteq Z(X)\cap\korin{\IN}{V\!E(X)}.$$

If $\C=\mathsf{T_{\!1}S}$, then by Theorem~\ref{t:mainT1}, the set $Z(X)\cap\korin{\IN}{V\!E(X)}$ is finite and so is its subset $I\!Z(X)$.

If $\mathsf{T_{\!z}S}\subseteq \C\subseteq \mathsf{T_{\!2}S}$, then the semigroup $I\!Z(X)$ is absolutely $\C$-closed by Corollary~\ref{c:IZ}.

\section{Proof of Theorem~\ref{t:main2}}\label{s:final}

Given a class $\C$ of topological semigroups with $\mathsf{T_{\!z}S}\subseteq\C\subseteq\mathsf{T_{\!2}S}$, and a commutative semigroup $X$, we shall prove the equivalence of the following conditions.
\begin{enumerate}
\item $X$ is absolutely $\C$-closed;
\item $X$ is ideally $\C$-closed, injectively $\C$-closed and bounded;
\item $X$ is ideally $\C$-closed, group-finite and bounded;
\item $X$ is chain-finite, bounded, group-finite and Clifford+finite.
\end{enumerate}

The implication $(1)\Ra(2)$ follows from  Corollary~\ref{c:IZ} and the equality $I\!Z(X)=Z(X)=X$ holding by the commutativity of $X$.

The implication $(2)\Ra(3)$, $(3)\Ra(4)$, and $(4)\Ra(1)$ follow from Theorems~\ref{t:Z}(2), \ref{t:mainP}, and Lemma~\ref{l:chain-closed2}, respectively.

\end{document}